\documentclass{article}%
\usepackage{amsfonts}
\usepackage{amsmath}
\usepackage{amssymb}
\usepackage{graphicx}%
\numberwithin{equation}{section}
\newtheorem{theorem}{Theorem}[section]

\newtheorem{corollary}[theorem]{Corollary}

\newtheorem{definition}[theorem]{Definition}

\newtheorem{lemma}[theorem]{Lemma}
\newtheorem{notation}[theorem]{Notation}

\newtheorem{remark}[theorem]{Remark}

\newenvironment{proof}[1][Proof]{\noindent\textbf{#1.} }{\ \rule{0.5em}{0.5em}}
\begin{document}

\title{Interior $HW^{1,p}$ estimates for divergence degenerate elliptic systems in
Carnot groups\thanks{This work was supported by the National Natural Science
Foundation of China (Grant Nos. 10871157 and 11001221), specialized Research
Found for the Doctoral Program of Higher Education (No. 200806990032) and
Northwestern Polytechnical University jichu yanjiu jijin tansuo xiangmu (No.
JC 201124). \textbf{2000 AMS\ Classification}: Primary 35H05. Secondary:
35J45, 35B65. \textbf{Keywords}: H\"{o}rmander's vector fields, subelliptic
systems, $L^{p}$ estimates. \textbf{Corresponding author} Marco Bramanti}}
\author{Maochun Zhu, Marco Bramanti, Pengcheng Niu}
\maketitle

\begin{abstract}
Let $X_{1},\ldots,X_{q}$ be the basis of the space of horizontal vector fields
on a homogeneous Carnot group $\mathbb{G=}\left(
\mathbb{R}
^{n},\circ\right)  $ ($q<n$). We consider the following divergence degenerate
elliptic system%
\[
\sum_{\beta=1}^{N}\sum_{i,j=1}^{q}X_{i}\left(  a_{\alpha\beta}^{ij}%
(x)X_{j}u^{\beta}\right)  =\sum_{i=1}^{q}X_{i}f_{\alpha}^{i},\text{
\ \ \ }\alpha=1,2,...,N
\]
where the coefficients $a_{\alpha\beta}^{ij}$ are real valued bounded
measurable functions defined in $\Omega\subset\mathbb{G}$, satisfying the
strong Legendre condition and belonging to the space $VMO_{loc}\left(
\Omega\right)  $ (defined by the Carnot-Carath\'{e}odory distance induced by
the $X_{i}$'s). We prove interior $HW^{1,p}$ estimates ($2\leq p<\infty$) for
weak solutions to the system.

\end{abstract}

\section{Introduction}

Let
\[
X_{i}=\sum_{j=1}^{n}b_{ij}\left(  x\right)  \partial_{x_{j}},\ i=1,2,...,q,
\]
be a family of real smooth vector fields defined in some bounded domain
$\Omega\subset$ $\mathbb{R}^{n}$ $\left(  q<n\right)  $ and satisfying
H\"{o}rmander's condition: the Lie algebra generated by $X_{1},...,X_{q}$
spans $\mathbb{R}^{n}$ at any point of $\Omega$. Since H\"{o}rmander's famous
paper \cite{h}, there has been tremendous work on the geometric properties of
H\"{o}rmander's vector fields, see \cite{nsw2}, \cite{J}, \cite{FL},
\cite{BGW}, \cite{garofalo}, \cite{G lu}, \cite{G lu1}, and references
therein. Meanwhile, regularity for linear degenerate elliptic equations
involving vector fields has been investigated and many results have been
proved, see for instance \cite{Folland}, \cite{rs}, \cite{BB1}, \cite{bb2}%
,\ \cite{bb4}, \cite{bb5}, \cite{BBLU}, \cite{BZ}, \cite{TN}, \cite{G lu},
\cite{G lu1} and references therein; as for subelliptic systems structured on
H\"{o}rmander's vector fields, we can quote \cite{Fazio}, \cite{xu-zu},
\cite{shores}.

In this paper we consider divergence degenerate elliptic systems structured on
H\"{o}rmander's vector fields in Carnot groups. Namely (here we briefly state
our assumptions and result; precise definitions and assumptions will be given
in \S \ \ref{sec background}), let $X_{1},\ldots,X_{q}$ be the canonical basis
of the space of horizontal vector fields in a homogeneous Carnot group
$\mathbb{G=}\left(  \mathbb{R}^{n},\circ\right)  $; we consider the system%
\begin{equation}
X_{i}\left(  a_{\alpha\beta}^{ij}(x)X_{j}u^{\beta}\right)  =X_{i}f_{\alpha
}^{i} \label{system}%
\end{equation}
in some domain $\Omega\subset\mathbb{R}^{n}$ where $\alpha,\beta
=1,\ldots,N,i,j=1,2,...,q$, $\mathbf{F}=\left(  f_{\alpha}^{i}\right)  \in
L^{p}\left(  \Omega;\mathbb{M}^{N\times q}\right)  $ $\left(  2\leq
p<\infty\right)  $ is a given $N\times q$ matrix. In (\ref{system}) and
throughout the paper, the summation is understood for repeated indices. If the
tensor $\left\{  a_{\alpha\beta}^{ij}(x)\right\}  $ satisfies the strong
Legendre condition (see (\ref{elliptic condition})), by Lax-Milgram theorem
the natural functional framework for solutions to (\ref{system}) is the
Sobolev space $HW_{loc}^{1,2}\left(  \Omega;\mathbb{R}^{N}\right)  $, so the
regularity problem for (\ref{system}) amounts to asking: if $\mathbf{F}\in
L^{p}\left(  \Omega;\mathbb{M}^{N\times q}\right)  $ for some $p>2$, can we
say that $u\in HW^{1,p}$, at least locally? We will prove an affirmative
answer to this question (see Theorem \ \ref{lp estimate}), under the
assumption that the coefficients $a_{\alpha\beta}^{ij}$ belong to the space
$VMO_{loc}\left(  \Omega\right)  $, with respect to the
Carnot--Carath\'{e}odory distance induced by the vector fields. Under this
respect, this result is in the same spirit as the $L^{p}$ regularity results
proved for nonvariational elliptic equations by Chiarenza-Frasca-Longo
\cite{cfl1}, \cite{cfl2}, for elliptic systems by Chiarenza-Franciosi-Frasca
\cite{CFF} (see also \cite{Chia}), and for nondivergence equations structured
on H\"{o}rmander's vector fields by Bramanti-Brandolini \cite{BB1},
\cite{bb2}, while analogous regularity estimates in Morrey spaces have been
proved for instance by Di Fazio-Palagachev-Ragusa in \cite{Fazio 2}, and by
Palagachev-Softova in \cite{Palagachev}. However, the technique of the proof
in the present case is completely different. Namely, while in all the
aforementioned papers $L^{p}$ or Morrey estimates are proved by exploiting
representation formulas for solutions and singular integral estimates, in the
case of subelliptic \emph{systems}, even on Carnot groups, no result about
representation formulas by means of homogeneous fundamental solutions seems to
be known. Hence we have to make use of a different technique, which has been
designed and exploited in a series of papers by Byun-Wang to deal with
elliptic equations and systems, also in very rough domains: see \cite{wang},
\cite{ss b1}, \cite{SS b} and references therein. Namely, the key technical
point is a series of local estimates involving the \emph{maximal function }of
$\left\vert Xu\right\vert ^{2}$ (\S \S \ \ref{sec maximal}%
-\ref{section L^p estimates}) which hold under an assumption of smallness of
the mean oscillation of the coefficients. One of the tools used to prove these
local estimates is the possibility of approximating, locally, the solution to
a system with small datum and small oscillation of the coefficients by the
solution to a different system, with \emph{constant }coefficients
(\S \ \ref{sec approx}). In turn, the solution to a constant coefficients
system on a Carnot group is known to satisfy an $L^{\infty}$ gradient bound
(see Theorem \ref{L2.4}) which turns out to be a key tool in our proof. This
result about systems with constant coefficients in Carnot groups has been
proved by Shores \cite{shores}, and represents one of the main reasons why we
have restricted ourselves to the case of Carnot groups instead of considering
general H\"{o}rmander's vector fields.

This paper represents the first case of study of $L^{p}$ estimates on the
\textquotedblleft subelliptic gradient\textquotedblright\ $Xu$ for subelliptic
systems. Di Fazio and Fanciullo in \cite{Fazio} have deduced interior Morrey
regularity in spaces $L^{2,\lambda}$ for weak solutions to the system
(\ref{system}) under the assumption that the coefficients $a_{\alpha\beta
}^{ij}$ belong to the class $VMO_{X}\cap L^{\infty\text{ }}$, while
Schauder-type estimates have been proved for subelliptic systems by Xu--Zuily
\cite{xu-zu}.

This paper is organized as follows: in Section 2 we recall some basic facts
about Carnot groups and state precisely our assumptions and main results; in
Section 3 we prove the approximation result for local solutions to the
original system by means of solutions to a system with constant coefficients;
in Section 4 we prove some local estimates on the Hardy-Littlewood maximal
function of $\left\vert Xu\right\vert ^{2}$, and in Section 5 we come to the
proof of our main result.

\section{Preliminaries and statement of the
results\label{Section Known results}}

\subsection{Background on Carnot groups\label{sec background}}

We are going to recall here a few facts about Carnot groups that we will need
in the following. For the proofs, more properties, and examples, we refer the
reader to the paper \cite{Folland}, the books \cite{BLU} and \cite[Chaps.
XII-XIII]{St}.

\begin{definition}
[Homogeneous Carnot groups]\label{Def Carnot group}A \emph{homogeneous group}
$\mathbb{G}$ is the set $\mathbb{R}^{n}$ endowed with a Lie group operation
$\circ$ (\textquotedblleft translation\textquotedblright), where the origin is
the group identity, and a family $\left\{  D\left(  \lambda\right)  \right\}
_{\lambda>0}$ of group automorphisms (\textquotedblleft
dilations\textquotedblright), acting as follows:%
\[
D\left(  \lambda\right)  \left(  x_{1},x_{2},...,x_{n}\right)  =\left(
\lambda^{\alpha_{1}}x_{1},\lambda^{\alpha_{2}}x_{2},...,\lambda^{\alpha_{n}%
}x_{n}\right)  \text{ \ \ }\forall\lambda>0
\]
for some fixed exponents $0<\alpha_{1}<\alpha_{2}<...<\alpha_{n}$. The number
$Q=\sum_{j=1}^{n}\alpha_{j}$ is called the \emph{homogeneous dimension} of
$\mathbb{G}$.

We say that a vector field $X=\sum_{j=1}^{n}b_{j}\left(  x\right)
\partial_{x_{j}}$ is left invariant if for any smooth function $f$ one has
\[
X^{x}\left(  f\left(  y\circ x\right)  \right)  =\left(  Xf\right)  \left(
y\circ x\right)  \text{ \ \ \ }\forall x,y\in\mathbb{G};
\]
we say that $X$ is $k$-homogeneous if for any smooth function $f$ one has
\[
X\left(  f\left(  D\left(  \lambda\right)  x\right)  \right)  =\lambda
^{k}\left(  Xf\right)  \left(  D\left(  \lambda\right)  x\right)
\text{\ \ }\forall\lambda>0,x\in\mathbb{G}.
\]

Let $X_{i}$ ($i=1,2,...,n$) be the unique left invariant vector field on
$\mathbb{G}$ which at the origin coincides with $\partial_{x_{i}}$. We assume
that for some integer $q<n$ the vector fields $X_{1},X_{2},...,X_{q}$ are
$1$-homogeneous and satisfy H\"{o}rmander's condition in $\mathbb{R}^{n}$: the
Lie algebra generated by the $X_{i}$'s at any point has dimension $n$. Under
these assumptions we say that $\mathbb{G}$ is a \emph{homogeneous Carnot
group} and that $\left\{  X_{1},X_{2},...,X_{q}\right\}  $ is the
\emph{canonical basis} of the space of horizontal vector fields.
\end{definition}

The properties required in the above definition have a number of consequences:
the exponents $\alpha_{i}$ are actually positive integers, the Lie algebra of
$\mathbb{G}$ is stratified, homogeneous and nilpotent; the vector fields
$X_{i}$ have polynomial coefficients. Moreover, the Lebesgue measure of
$\mathbb{R}^{n}$ is the Haar measure in $\mathbb{G}$.

Like for any set of H\"{o}rmander's vector fields, it is possible to define
the corresponding \emph{Carnot-Carath\'{e}odory distance} $d_{X},$ as follows.

\begin{definition}
[CC-distance]\label{Def distance}For any $\delta>0,$ let $C_{\delta}$ be the
set of absolutely continuous curves $\phi:\left[  0,1\right]  \rightarrow
\mathbb{R}^{n}$ such that%
\[
\phi^{\prime}\left(  t\right)  =\sum_{i=1}^{q}a_{i}\left(  t\right)
X_{i}\left(  \phi_{i}\left(  t\right)  \right)  \text{ with }\left\vert
a_{i}\left(  t\right)  \right\vert \leq\delta\text{ for a.e. }t\in\left[
0,1\right]  .
\]
Then%
\[
d_{X}\left(  x,y\right)  =\inf\left\{  \delta>0:\exists\phi\in C_{\delta
}\text{ with }\phi\left(  0\right)  =x,\phi\left(  1\right)  =y\right\}  .
\]

\end{definition}

The function $d_{X}$ turns out to be finite for any couple of points, and is
actually a distance, called Carnot-Carath\'{e}odory distance; due to the
structure of Carnot group, $d_{X}$ is also left invariant and $1$-homogeneous
on $\mathbb{G}$. Let%

\[
B_{r}\left(  x\right)  =\left\{  y\in\mathbb{G}:d_{X}\left(  x,y\right)
<r\right\}
\]
be the metric ball of center $x$ and radius $r$ in $\mathbb{G}$. Since the
Lebesgue measure in $\mathbb{R}^{n}$ is the Haar measure on $\mathbb{G}$, one
has (writing $\left\vert A\right\vert $ for the measure of $A$)%
\begin{equation}
\left\vert B_{r}\left(  x\right)  \right\vert =\omega_{\mathbb{G}}r^{Q},
\label{dimension}%
\end{equation}
where $Q\ $is the homogeneous dimension of $\mathbb{G}$ and $\omega
_{\mathbb{G}}$ is a positive constant.

Next, we need to define the function spaces we will use in the following.

\begin{definition}
[Horizontal Sobolev spaces]\label{D2.1}For any $p\geq1$ and domain $\Omega$
$\subset\mathbb{G}$, let us define the \emph{horizontal Sobolev space}:%
\[
HW^{1,p}\left(  \Omega;\mathbb{R}^{N}\right)  =\left\{  u\in L^{p}\left(
\Omega;\mathbb{R}^{N}\right)  :\left\Vert u\right\Vert _{HW^{1,p}\left(
\Omega;%
\mathbb{R}
^{N}\right)  }<\infty\right\}
\]
with the norm%
\[
\left\Vert u\right\Vert _{HW^{1,p}\left(  \Omega;%
\mathbb{R}
^{N}\right)  }=\left\Vert u\right\Vert _{L^{p}\left(  \Omega;%
\mathbb{R}
^{N}\right)  }+\left\Vert Xu\right\Vert _{L^{p}\left(  \Omega;\mathbb{R}%
^{N}\right)  },
\]
having set%
\[%
\begin{array}
[c]{lll}%
\left\Vert u\right\Vert _{L^{p}\left(  \Omega;\mathbb{R}^{N}\right)
}=\left\Vert \left\vert u\right\vert \right\Vert _{L^{p}\left(  \Omega\right)
}, & \text{with} & \left\vert u\right\vert =\left(  \sum_{\alpha=1}%
^{N}\left\vert u^{\alpha}\right\vert ^{2}\right)  ^{1/2}\text{ and}\\
\left\Vert Xu\right\Vert _{L^{p}\left(  \Omega;\mathbb{R}^{N}\right)
}=\left\Vert \left\vert Xu\right\vert \right\Vert _{L^{p}\left(
\Omega\right)  }, & \text{with} & \left\vert Xu\right\vert =\left(
\sum_{\alpha=1}^{N}\sum_{i=1}^{q}\left\vert X_{i}u^{\alpha}\right\vert
^{2}\right)  ^{1/2}.
\end{array}
\]
Also, we define the space $HW_{loc}^{1,p}\left(  \Omega;\mathbb{R}^{N}\right)
$ as the space of functions $u$ such that $u\phi\in HW^{1,p}\left(
\Omega;\mathbb{R}^{N}\right)  $ for any $\phi\in C_{0}^{\infty}\left(
\Omega\right)  $ and the space $HW_{0}^{1,p}\left(  \Omega;\mathbb{R}%
^{N}\right)  $ as the closure of $C_{0}^{\infty}\left(  \Omega;\mathbb{R}%
^{N}\right)  $ in the norm $HW^{1,p}\left(  \Omega;\mathbb{R}^{N}\right)  $.
\end{definition}

\begin{definition}
[BMO-type spaces]\label{D2.2}For any $\Omega^{\prime}\Subset\Omega,$ let
$R_{0}$ be a number such that $B_{r}\left(  x\right)  \Subset\Omega$ for any
$x\in\Omega^{\prime}$ and $r\leq R_{0}$. For any $f\in L_{loc}^{1}\left(
\Omega\right)  $ and $r\leq R_{0},$ let%
\[
\eta_{\Omega^{\prime},R_{0},f}\left(  r\right)  =\underset{x_{0}\in
\Omega^{\prime},0<\rho\leq r}{\sup}\frac{1}{\left\vert B_{\rho}\left(
x_{0}\right)  \right\vert }\int_{B_{\rho}\left(  x_{0}\right)  }\left\vert
f\left(  x\right)  -f_{B_{\rho}\left(  x_{0}\right)  }\right\vert ^{2}dx,
\]
where $f_{B_{\rho}\left(  x_{0}\right)  }=\frac{1}{\left\vert B_{\rho}\left(
x_{0}\right)  \right\vert }\int_{B_{\rho}\left(  x_{0}\right)  }f\left(
x\right)  dx$.

We say that $f$ is ($\delta,R$)-vanishing in $\Omega^{\prime}$ (for a couple
of fixed positive numbers $\delta,R$, with $R\leq R_{0}$) if
\[
\eta_{\Omega^{\prime},R_{0},f}\left(  R\right)  <\delta^{2}.
\]

We say that $f\in VMO_{loc}\left(  \Omega\right)  $ if for any $\Omega
^{\prime}\Subset\Omega$ and $R_{0}$ such that $B_{r}\left(  x\right)
\Subset\Omega$ for any $r\leq R_{0}$ and $x\in\Omega^{\prime},$ we have%
\[
\eta_{\Omega^{\prime},R_{0},f}\left(  r\right)  \rightarrow0\text{ as
}r\rightarrow0.
\]
The function $\eta_{\Omega^{\prime},R_{0},f}$ is called the local $VMO$
modulus of $f$ on $\Omega^{\prime}$.
\end{definition}

We will use the following well-known result by Jerison (see \cite[Thm 2.1]{J}
for the case $p=2$ and \cite[\S 6]{J} for $p\neq2$):

\begin{theorem}
[Poincar\'{e}'s inequality]For $1\leq p<\infty$ there exists a positive
constant $c=c\left(  \mathbb{G},p\right)  $, such that for any $u\in
HW^{1,p}\left(  B_{R}\right)  ,$
\begin{equation}
\left\Vert u-u_{B_{R}}\right\Vert _{L^{p}\left(  B_{R}\right)  }\leq
cR\left\Vert Xu\right\Vert _{L^{p}\left(  B_{R}\right)  }.
\label{poincare inequality}%
\end{equation}
If $u\in HW_{0}^{1,p}\left(  B_{R}\right)  $,
\begin{equation}
\left\Vert u\right\Vert _{L^{p}\left(  B_{R}\right)  }\leq cR\left\Vert
Xu\right\Vert _{L^{p}\left(  B_{R}\right)  }. \label{poincare inequality 1}%
\end{equation}

\end{theorem}

The previous theorem holds for a general system of H\"{o}rmander's vector
fields; in that case, however, some restriction on the center and radius of
the ball $B_{R}$ applies (see \cite[Thm. 2.1]{J}); on a Carnot group, instead,
due to the dilation invariance of the inequalities (\ref{poincare inequality})
and (\ref{poincare inequality 1}), these hold for any ball $B_{R}$ and with an
\textquotedblleft absolute\textquotedblright\ constant $c$.

We will also make use of the following

\begin{definition}
[Space of homogeneous type, see \cite{cw}]Let $S$ be a set and $d:$ $S\times
S\rightarrow$ $\left[  0,\infty\right)  $ a quasidistance, that is, for some
constant $c\geq1$ one has%
\begin{align}
d\left(  x,y\right)   &  =0\Longleftrightarrow x=y\nonumber\\
d\left(  x,y\right)   &  =d\left(  y,x\right) \nonumber\\
d\left(  x,y\right)   &  \leq c\left[  d\left(  x,z\right)  +d\left(
z,y\right)  \right]  \label{quasitriangle}%
\end{align}
for all $x,y,z\in S$. The balls defined by $d$ induce a topology in $S$; let
us assume that the $d$-balls are open in this topology. Moreover, assume that
there exists a regular Borel measure $\mu$\ on $S$, such that the "doubling
condition" is satisfied:
\begin{equation}
\mu\left(  B_{2r}(x)\right)  \leq c\mu\left(  B_{r}(x)\right)  ,
\label{doubling}%
\end{equation}
for every $r>0$, $x\in S$ and some positive constant $c$. Then we say that
$(S,d,\mu)$ is a \emph{space of homogeneous type}.
\end{definition}

\begin{remark}
\label{d-regular and homegeneous}Note that in our context any
Carnot-Carath\'{e}odory ball $B_{R}\left(  x_{0}\right)  $ is a $d_{X}%
$-regular domain (see for instance \cite[Lemma 4.2]{bb4}), that is there
exists a positive constant $c_{d}$ such that
\begin{equation}
\left\vert B_{R}\left(  x_{0}\right)  \cap B_{r}\left(  x\right)  \right\vert
\geq c_{d}\left\vert B_{r}\left(  x\right)  \right\vert \text{ \ }\forall
r>0,\forall x\in B_{R}\left(  x_{0}\right)  . \label{d-regular}%
\end{equation}
This implies that $\left(  B_{R}\left(  x_{0}\right)  ,d_{X},dx\right)  $ is a
space of homogeneous type. Moreover, a simple dilation argument shows that, in
a Carnot group, the constant $c_{d},$ and therefore the doubling constant of
$\left(  B_{R}\left(  x_{0}\right)  ,d_{X},dx\right)  $, is independent of $R$.
\end{remark}

\subsection{Assumptions and known results about degenerate
systems\label{Assumption}}

The general assumptions which will be in force throughout the paper are
collected in the following:

\textbf{Assumption (H). }We assume that $\mathbb{G}$ is a homogeneous Carnot
group in $\mathbb{R}^{n}$ and $\left\{  X_{1},X_{2},...,X_{q}\right\}  $ is
the canonical basis of the space of horizontal vector fields in $\mathbb{G}$
(see Definition \ref{Def Carnot group}). We assume that the coefficients
\[
\left\{  a_{\alpha\beta}^{ij}\right\}  _{\substack{i,j=1,...,q\\\alpha
,\beta=1,...,N}}
\]
in (\ref{system}) are real valued, bounded measurable functions defined in
$\Omega$ and satisfying the \emph{strong Legendre condition}: there exists a
constant $\mu>0$ such that
\begin{equation}
\mu|\xi|^{2}\leq a_{\alpha\beta}^{ij}(x)\xi_{i}^{\alpha}\xi_{j}^{\beta}\leq
\mu^{-1}|\xi|^{2}\label{elliptic condition}%
\end{equation}
for any $\xi\in\mathbb{M}^{N\times q}$ , a.e.$\,\ x\in\Omega$.

\bigskip

We recall the standard definition of weak solution:

\begin{definition}
We say that $u\in HW^{1,2}\left(  \Omega;\mathbb{R}^{N}\right)  $ is a weak
solution to the system (\ref{system}), if it satisfies%
\[
\int_{\Omega}a_{\alpha\beta}^{ij}(x)X_{j}u^{\beta}X_{i}\varphi^{\alpha}%
dx=\int_{\Omega}f_{\alpha}^{i}X_{i}\varphi^{\alpha}dx
\]
for any $\varphi\in HW_{0}^{1,2}\left(  \Omega;\mathbb{R}^{N}\right)  $.
\end{definition}

Recall that on a Carnot group the transposed of a vector field is just the
opposite: $X_{i}^{\ast}=-X_{i}$. Hence the above definition of weak solution
is consistent with the way the system (\ref{system}) is written.

\begin{remark}
\label{existence}Let $B_{R}\subset\Omega$ be any metric ball. If $f_{\alpha
}^{i}\in L^{2}\left(  B_{R}\right)  $ and $u_{0}\in HW^{1,2}\left(
B_{R}\right)  $, then by assumption (\ref{elliptic condition}) and
Poincar\'{e}'s inequality (\ref{poincare inequality 1}) we can apply
Lax-Milgram's theorem, and conclude that there exists a unique solution $u\in
HW^{1,2}\left(  B_{R};\mathbb{R}^{N}\right)  $ to system (\ref{system}) such
that $u-u_{0}\in HW_{0}^{1,2}\left(  B_{R}\right)  $. Moreover, the following
a priori estimate holds:%
\begin{equation}
\left\Vert u\right\Vert _{HW^{1,2}\left(  B_{R};\mathbb{R}^{N}\right)  }\leq
c\left(  \left\Vert \mathbf{F}\right\Vert _{L^{2}(B_{R};\mathbb{M}^{N\times
q})}+\left\Vert u_{0}\right\Vert _{HW^{1,2}\left(  B_{R};\mathbb{R}%
^{N}\right)  }\right)  \label{a priori Hilbert}%
\end{equation}
for some constant $c$ only depending on $\mathbb{G},\mu,R$ (see \cite[Chap.
8]{c.w} for a proof of this fact in the elliptic case).
\end{remark}

The next result is taken from \cite[Corollary 19]{shores}. See also
\cite{EMS}, where the analogous parabolic inequality is proved.

\begin{theorem}
\label{L2.4}Let $v\in HW^{1,2}\left(  B\left(  x_{0},KR\right)  ;\mathbb{R}%
^{N}\right)  $ be a solution to the system%
\[
X_{i}\left(  a_{\alpha\beta}^{ij}X_{j}v^{\beta}\right)  =0\text{ in }B\left(
x_{0},KR\right)
\]
with constant coefficients $a_{\alpha\beta}^{ij}$ satisfying
(\ref{elliptic condition}) and some $K>1$. Then $v\in C^{\infty}\left(
B\left(  x_{0},KR\right)  ;\mathbb{R}^{N}\right)  $; moreover \ \
\[
\underset{B_{R}\left(  x_{0}\right)  }{\sup}\left\vert Xv\right\vert ^{2}\leq
cR^{-2}\frac{1}{\left\vert B_{KR}\left(  x_{0}\right)  \right\vert }%
\int_{B_{KR}\left(  x_{0}\right)  }\left\vert v\right\vert ^{2}dx,
\]
where the positive constant $c$ depends on $K,\mu,\mathbb{G},N$ but is
independent of $x_{0}$, $R$ and $v$.
\end{theorem}

The following result can be proved in a completely standard way by suitable
cutoff functions (for the analogous elliptic version see for instance
\cite[Thm. 2.1 p.134]{c.w}):

\begin{theorem}
[Caccioppoli's inequality]Let $u\in HW^{1,2}\left(  B_{R}\left(  \overline
{x}\right)  ;\mathbb{R}^{N}\right)  $ be a weak solution to (\ref{system}) in
$B_{R}\left(  \overline{x}\right)  \subset\Omega$. There exists a constant
$c>0$ depending on $\mathbb{G},N,R$ such that for any $\rho\in\left(
0,R\right)  ,$%
\begin{equation}
\int_{B_{\rho}\left(  \overline{x}\right)  }\left\vert Xu\left(  x\right)
\right\vert ^{2}dx\leq c\left[  \frac{1}{\left(  R-\rho\right)  ^{2}}%
\int_{B_{R}\left(  \overline{x}\right)  }\left\vert u\left(  x\right)
\right\vert ^{2}dx+\int_{B_{R}\left(  \overline{x}\right)  }\left\vert
\mathbf{F}\left(  x\right)  \right\vert ^{2}dx\right]  . \label{Caccioppoli}%
\end{equation}

\end{theorem}

\subsection{Statement of the result\label{sec main result}}

We now state precisely the main result of this paper.

\begin{theorem}
\label{lp estimate}Under the Assumption (H),\textbf{ }let the $a_{\alpha\beta
}^{ij}$'s belong to $VMO_{loc}\left(  \Omega\right)  $ and let $\Omega
^{\prime}\Subset\Omega$, $2<p<\infty$. Then there is a positive constant $c$
depending on $\mathbb{G},\mu,p,\Omega,\Omega^{\prime}$ and the local $VMO$
moduli of the $a_{\alpha\beta}^{ij}$'s in $\Omega^{\prime}$ such that if
$\mathbf{F}=\left(  f_{\alpha}^{i}\right)  \in L^{p}\left(  \Omega
;\mathbb{M}^{N\times q}\right)  $ and $u\in HW^{1,2}\left(  \Omega
;\mathbb{R}^{N}\right)  $ is a weak solution to (\ref{system}) in $\Omega$,
then $u\in HW^{1,p}\left(  \Omega^{\prime};\mathbb{R}^{N}\right)  $ and%
\begin{equation}
\left\Vert u\right\Vert _{HW^{1,p}\left(  \Omega^{\prime};\mathbb{R}%
^{N}\right)  }\leq c\left(  \left\Vert \mathbf{F}\right\Vert _{L^{p}%
(\Omega;\mathbb{M}^{N\times q})}+\left\Vert u\right\Vert _{L^{2}\left(
\Omega;\mathbb{R}^{N}\right)  }\right)  . \label{main estimate}%
\end{equation}

\end{theorem}

In order to prove Theorem \ref{lp estimate}, we will prove the following local result:

\begin{theorem}
\label{Thm main basic}Under the Assumption (H), for any $\overline{x}\in
\Omega,R_{0}>0$ such that $B_{11R_{0}}\left(  \overline{x}\right)
\subset\Omega$ there exists $\delta=\delta\left(  p,\mathbb{G},R_{0}%
,\mu\right)  >0$ such that for any $R\leq R_{0}$, if the coefficients
$a_{\alpha\beta}^{ij}$ are $\left(  \delta,8R\right)  $-vanishing in
$B_{R}\left(  \overline{x}\right)  $ and $p\in\left(  2,\infty\right)  $, then
there is a positive $c=c\left(  R,R_{0},p,\mathbb{G}\right)  $ such that if
$\mathbf{F}=\left(  f_{\alpha}^{i}\right)  \in L^{p}\left(  B_{11R}\left(
\overline{x}\right)  ;\mathbb{M}^{N\times q}\right)  $ and $u\in
HW^{1,2}\left(  B_{11R}\left(  \overline{x}\right)  ;\mathbb{R}^{N}\right)  $
is a weak solution of (\ref{system}) in $B_{11R}\left(  \overline{x}\right)
$, then $u\in HW^{1,p}\left(  B_{R}\left(  \overline{x}\right)  ;\mathbb{R}%
^{N}\right)  $ and%
\begin{equation}
\left\Vert Xu\right\Vert _{L^{p}\left(  B_{R}\left(  \overline{x}\right)
;\mathbb{R}^{N}\right)  }\leq c\left(  \left\Vert \mathbf{F}\right\Vert
_{L^{p}(B_{11R}\left(  \overline{x}\right)  ;\mathbb{M}^{N\times q}%
)}+\left\Vert Xu\right\Vert _{L^{2}(B_{11R}\left(  \overline{x}\right)
;\mathbb{R}^{N})}\right)  . \label{local L^p}%
\end{equation}

\end{theorem}

\begin{proof}
[Proof of Theorem \ref{lp estimate} from Theorem \ref{Thm main basic}]For
fixed domains $\Omega^{\prime}\Subset\Omega^{\prime\prime}\Subset\Omega$, pick
$R_{0}$ such that $B_{12R_{0}}\left(  \overline{x}\right)  \subset
\Omega^{\prime\prime}$ for any $\overline{x}\in\Omega^{\prime}$. For this
$R_{0}$ and a fixed $p\in\left(  2,\infty\right)  ,$ let $\delta$ be like in
Theorem \ref{Thm main basic}. Since the $a_{\alpha\beta}^{ij}$'s belong to
$VMO_{loc}\left(  \Omega\right)  $, there exists $R\leq R_{0}$, $R$ depending
on $\Omega$, $\Omega^{\prime}$, $R_{0}$, $\delta,$ such that the
$a_{\alpha\beta}^{ij}$'s are $\left(  \delta,8R\right)  $-vanishing in
$B_{R}\left(  \overline{x}\right)  $. Therefore by Theorem
\ref{Thm main basic}, (\ref{local L^p}) holds for any such $\overline{x}$ and
$R$. Next, we apply Caccioppoli's inequality (\ref{Caccioppoli}), getting%
\[
\left\Vert Xu\right\Vert _{L^{2}(B_{11R}\left(  \overline{x}\right)
;\mathbb{R}^{N})}\leq c\left\{  \frac{1}{R}\left\Vert u\right\Vert
_{L^{2}(B_{12R}\left(  \overline{x}\right)  ;\mathbb{R}^{N})}+\left\Vert
\mathbf{F}\right\Vert _{L^{2}(B_{12R}\left(  \overline{x}\right)
;\mathbb{M}^{N\times q})}\right\}  ,
\]
which inserted in (\ref{local L^p}) gives%
\begin{equation}
\left\Vert Xu\right\Vert _{L^{p}\left(  B_{R}\left(  \overline{x}\right)
;\mathbb{R}^{N}\right)  }\leq c\left(  R\right)  \left\{  \left\Vert
\mathbf{F}\right\Vert _{L^{p}(B_{12R}\left(  \overline{x}\right)
;\mathbb{M}^{N\times q})}+\left\Vert u\right\Vert _{L^{2}(B_{12R}\left(
\overline{x}\right)  ;\mathbb{R}^{N})}\right\}  . \label{local 1}%
\end{equation}

On the other hand, by Poincar\'{e} inequality (\ref{poincare inequality}) we
have:%
\begin{align*}
\left\Vert u^{\alpha}\right\Vert _{L^{p}\left(  B_{R}\left(  \overline
{x}\right)  \right)  }  &  \leq\left\Vert u^{\alpha}-u_{B_{R}\left(
\overline{x}\right)  }^{\alpha}\right\Vert _{L^{p}\left(  B_{R}\left(
\overline{x}\right)  \right)  }+\left\vert u_{B_{R}\left(  \overline
{x}\right)  }^{\alpha}\right\vert \left\vert B_{R}\left(  \overline{x}\right)
\right\vert ^{1/p}\\
&  \leq cR\left\Vert Xu^{\alpha}\right\Vert _{L^{p}\left(  B_{R}\left(
\overline{x}\right)  \right)  }+\left\Vert u^{\alpha}\right\Vert
_{L^{2}\left(  B_{R}\left(  \overline{x}\right)  \right)  }\left\vert
B_{R}\left(  \overline{x}\right)  \right\vert ^{1/p-1/2},
\end{align*}
hence%
\[
\left\Vert u\right\Vert _{L^{p}\left(  B_{R}\left(  \overline{x}\right)
;\mathbb{R}^{N}\right)  }\leq c\left(  R,p\right)  \left\{  \left\Vert
Xu\right\Vert _{L^{p}\left(  B_{R}\left(  \overline{x}\right)  ;\mathbb{R}%
^{N}\right)  }+\left\Vert u\right\Vert _{L^{2}\left(  B_{R}\left(
\overline{x}\right)  ;\mathbb{R}^{N}\right)  }\right\}  ,
\]
which together with (\ref{local 1}) gives
\[
\left\Vert u\right\Vert _{HW^{1,p}\left(  B_{R}\left(  \overline{x}\right)
;\mathbb{R}^{N}\right)  }\leq c\left\{  \left\Vert \mathbf{F}\right\Vert
_{L^{p}(B_{12R}\left(  \overline{x}\right)  ;\mathbb{M}^{N\times q}%
)}+\left\Vert u\right\Vert _{L^{2}(B_{12R}\left(  \overline{x}\right)
;\mathbb{R}^{N})}\right\}  .
\]
A covering argument then gives (\ref{main estimate}).
\end{proof}

It is worthwhile to point out that, as we will see from the proof of Theorem
\ref{Thm main basic} in \S \ \ref{section L^p estimates}, the following bound,
stronger than (\ref{local L^p}), is actually established:%
\begin{equation}
\left\Vert \mathcal{M}_{B_{11R}\left(  \overline{x}\right)  }\left(
\left\vert Xu\right\vert ^{2}\right)  \right\Vert _{L^{p/2}\left(
B_{R}\left(  \overline{x}\right)  ;\mathbb{R}^{N}\right)  }^{1/2}\leq
c\left\{  \left\Vert \mathbf{F}\right\Vert _{L^{p}(B_{12R}\left(  \overline
{x}\right)  ;\mathbb{M}^{N\times q})}+\left\Vert u\right\Vert _{L^{2}%
(B_{12R}\left(  \overline{x}\right)  ;\mathbb{R}^{N})}\right\}  ,
\label{maximal lp}%
\end{equation}
where $\mathcal{M}$ is the Hardy-Littewood maximal function (see
\S \ref{sec maximal}).

\begin{remark}
Note that what allows to exploit the VMO assumption on the coefficients is the
fact that the number $\delta$ in Theorem \ref{Thm main basic} depends on
$R_{0}$ but not on $R\leq R_{0},$ which allows shrinking $R$ without changing
$\delta,$ to get the $\left(  \delta,R\right)  $-vanishing condition
satisfied. Under this regard, our result is very different from those proved
for instance in \cite{SS b}, \cite{ss b1} where the parameter $\delta$
possibly depends on $R$, which makes the $\left(  \delta,R\right)  $-vanishing
assumption hard to check.
\end{remark}

\textbf{Dependence of constants. }Throughout this paper, the letter $c$
denotes a constant which may vary from line to line. The parameters which the
constants depend on are declared in the statements or in the proofs of the
theorems. When we write that $c$ is an \textquotedblleft absolute
constant\textquotedblright\ we mean that it may depend on $\mathbb{G}$ and $N$.

\section{Approximation by solutions of systems with constant
coefficients\label{sec approx}}

\begin{notation}
In order to simplify notation, henceforth we will systematically write the
norms and spaces of vector valued functions as%
\begin{align*}
&  HW^{1,p}\left(  B\right)  ,\left\Vert u\right\Vert _{HW^{1,p}\left(
B\right)  },\left\Vert \mathbf{F}\right\Vert _{L^{p}\left(  B\right)  }\text{
instead of }\\
&  HW^{1,p}\left(  B;\mathbb{R}^{N}\right)  ,\left\Vert u\right\Vert
_{HW^{1,p}\left(  B;\mathbb{R}^{N}\right)  },\left\Vert \mathbf{F}\right\Vert
_{L^{p}\left(  B;\mathbb{M}^{N\times q}\right)  }\text{,}%
\end{align*}
and so on.
\end{notation}

In this section we will prove a couple of theorems asserting that a solution
to a system (\ref{system}) with small datum $\mathbf{F}$ and coefficients with
small oscillation, can be suitably approximated by a solution to a system with
constant coefficients and zero datum. This approximation is one of the tools
which will be used in the proof of Theorem \ref{Thm main basic}.

\begin{theorem}
\label{aproximate for u}Under Assumption (H) (see \S \ref{Assumption}), for
any $\varepsilon>0,$ $R_{0}>0$ there is a small $\delta=\delta\left(
\varepsilon,R_{0},\mu\right)  >0$ such that for any $R\leq R_{0}$, if
$u\,\ $is a weak solution$\ $to the system (\ref{system}) in $B_{4R}%
\Subset\Omega$ with%
\begin{equation}
\frac{1}{\left\vert B_{4R}\right\vert }\int_{B_{4R}}\left\vert Xu\right\vert
^{2}dx\leq1,\text{ \ }\frac{1}{\left\vert B_{4R}\right\vert }\int_{B_{4R}%
}\left(  \left\vert \mathbf{F}\right\vert ^{2}+\left\vert a_{\alpha\beta}%
^{ij}-\left(  a_{\alpha\beta}^{ij}\right)  _{B_{4R}}\right\vert ^{2}\right)
dx\leq\delta^{2}, \label{3.1}%
\end{equation}
then there exists a weak solution $v$ to the following homogeneous system with
constant coefficients:
\begin{equation}
X_{i}\left(  \left(  a_{\alpha\beta}^{ij}\right)  _{B_{4R}}X_{j}v^{\beta
}(x)\right)  =0\text{ in }B_{4R} \label{3.2}%
\end{equation}
such that%
\[
\frac{1}{R^{2}}\frac{1}{\left\vert B_{4R}\right\vert }\int_{B_{4R}}\left\vert
u-v\right\vert ^{2}dx\leq\varepsilon^{2}.
\]

\end{theorem}

\begin{proof}
Let us first prove the result for a fixed $R$ (and $\delta$ possibly depending
on $R$), then we will show how to remove the dependence on $R$.

By contradiction. If the result does not hold, then there exist a constant
$\varepsilon_{0}>0$, and sequences $\left\{  a_{\alpha\beta}^{ijk}\right\}
_{k=1}^{\infty}$ satisfying (\ref{elliptic condition}), $\left\{
u_{k}\right\}  _{k=1}^{\infty}$, $\left\{  \mathbf{F}_{k}\right\}
_{k=1}^{\infty}$ such that $u_{k}$ is a weak solution to the system%
\begin{equation}
X_{i}\left(  a_{\alpha\beta}^{ijk}X_{j}u_{k}^{\beta}(x)\right)  =X_{i}%
f_{\alpha}^{ik}\left(  x\right)  \text{ } \label{equation of aproximate}%
\end{equation}
in $B_{4R}$ with
\begin{equation}
\frac{1}{\left\vert B_{4R}\right\vert }\int_{B_{4R}}\left\vert Xu_{k}%
\right\vert ^{2}dx\leq1,\frac{1}{\left\vert B_{4}\right\vert }\int_{B_{4R}%
}\left(  \left\vert \mathbf{F}_{k}\right\vert ^{2}+\left\vert a_{\alpha\beta
}^{ijk}-\left(  a_{\alpha\beta}^{ijk}\right)  _{B_{4R}}\right\vert
^{2}\right)  dx\leq\frac{1}{k^{2}}, \label{bounded}%
\end{equation}
but
\begin{equation}
\frac{1}{R^{2}}\frac{1}{\left\vert B_{4R}\right\vert }\int_{B_{4R}}\left\vert
u_{k}-v_{k}\right\vert ^{2}dx>\varepsilon_{0}^{2} \label{3.6}%
\end{equation}
for any weak solution $v_{k}$ of
\begin{equation}
X_{i}\left(  \left(  a_{\alpha\beta}^{ijk}\right)  _{B_{4R}}X_{j}v_{k}^{\beta
}(x)\right)  =0\text{ in }B_{4R}. \label{system const}%
\end{equation}

From (\ref{bounded}) and Poincar\'{e}'s inequality (\ref{poincare inequality}%
), we know that $\left\{  u_{k}-\left(  u_{k}\right)  _{B_{4R}}\right\}
_{k=1}^{\infty}$ is bounded in $HW^{1,2}\left(  B_{4R}\right)  $, then
Rellich's lemma allows us to find a subsequence of $\left\{  u_{k}-\left(
u_{k}\right)  _{B_{4R}}\right\}  $, still denoted by $\left\{  u_{k}-\left(
u_{k}\right)  _{B_{4R}}\right\}  $, such that%
\begin{align}
\frac{1}{R^{2}}\frac{1}{\left\vert B_{4R}\right\vert }\int_{B_{4R}}\left\vert
u_{k}-\left(  u_{k}\right)  _{B_{4R}}-u_{0}\right\vert ^{2}dx  &
\rightarrow0\text{,}\label{3.8}\\
Xu_{k}  &  \rightarrow Xu_{0}\text{ weakly in }L^{2}\text{, } \label{add1}%
\end{align}
as $k\rightarrow\infty$, for some $u_{0}\in HW^{1,2}\left(  B_{4R}\right)  $.
Since $\left\{  \left(  a_{\alpha\beta}^{ijk}\right)  _{B_{4R}}\right\}
_{k=1}^{\infty}$ is bounded in $\mathbb{R}$, it allows a subsequence, still
denoted by $\left\{  \left(  a_{\alpha\beta}^{ijk}\right)  _{B_{4R}}\right\}
_{k=1}^{\infty}$, such that%
\begin{equation}
\left\vert \left(  a_{\alpha\beta}^{ijk}\right)  _{B_{4R}}-\bar{a}%
_{\alpha\beta}^{ij}\right\vert \rightarrow0,\text{ as }k\rightarrow\infty,
\label{3.9}%
\end{equation}
for some constants $\bar{a}_{\alpha\beta}^{ij}$. By (\ref{bounded}), it
follows%
\[
a_{\alpha\beta}^{ijk}\rightarrow\bar{a}_{\alpha\beta}^{ij}\text{ in }%
L^{2}\left(  B_{4R}\right)  ,\text{ as }k\rightarrow\infty.
\]

Next, we show that $u_{0}$ is a weak solution of%
\begin{equation}
X_{i}\left(  \bar{a}_{\alpha\beta}^{ij}X_{j}u^{\beta}(x)\right)  =0\text{ in
}B_{4R}. \label{3.11}%
\end{equation}
We start from%
\begin{equation}
\int_{B_{4R}}a_{\alpha\beta}^{ijk}\left(  x\right)  X_{j}u_{k}^{\beta}%
X_{i}\varphi^{\alpha}dx=\int_{B_{4R}}f_{\alpha}^{ik}X_{i}\varphi^{\alpha}dx
\label{3.12}%
\end{equation}
with $\varphi^{\alpha}\in C_{0}^{\infty}\left(  \Omega\right)  $, and take the
limit for $k\rightarrow\infty.$ By (\ref{bounded}),
\[
\int_{B_{4R}}f_{\alpha}^{ik}X_{i}\varphi^{\alpha}dx\rightarrow0.
\]
Moreover,%
\begin{align*}
&  \int_{B_{4R}}a_{\alpha\beta}^{ijk}\left(  x\right)  X_{j}u_{k}^{\beta}%
X_{i}\varphi^{\alpha}dx=\\
&  =\int_{B_{4R}}\left[  a_{\alpha\beta}^{ijk}\left(  x\right)  -\bar
{a}_{\alpha\beta}^{ij}\right]  X_{j}u_{k}^{\beta}X_{i}\varphi^{\alpha}%
dx+\int_{B_{4R}}\bar{a}_{\alpha\beta}^{ij}X_{j}u_{k}^{\beta}X_{i}%
\varphi^{\alpha}dx\equiv A_{k}+B_{k}.
\end{align*}
Now,%
\[
\left\vert A_{k}\right\vert \leq c\left\Vert a_{\alpha\beta}^{ijk}\left(
x\right)  -\bar{a}_{\alpha\beta}^{ij}\right\Vert _{L^{2}\left(  B_{4R}\right)
}\left\Vert X_{j}u_{k}^{\beta}\right\Vert _{L^{2}\left(  B_{4R}\right)
}\rightarrow0,
\]
because $a_{\alpha\beta}^{ijk}\left(  x\right)  \rightarrow\bar{a}%
_{\alpha\beta}^{ij}$ in $L^{2}$ and $\left\{  X_{j}u_{k}^{\beta}\right\}  $ is
bounded in $L^{2}$. Finally, since $Xu_{k}\rightarrow Xu_{0}$ weakly in
$L^{2},$%
\[
B_{k}\rightarrow\int_{B_{4R}}\bar{a}_{\alpha\beta}^{ij}X_{j}u_{0}^{\beta}%
X_{i}\varphi^{\alpha}dx,
\]
hence%
\[
\int_{B_{4R}}\bar{a}_{\alpha\beta}^{ij}X_{j}u_{0}^{\beta}X_{i}\varphi^{\alpha
}dx=0\text{ for any }\varphi^{\alpha}\in C_{0}^{\infty}\left(  B_{4R}\right)
.
\]
By density, this holds for any $\varphi^{\alpha}\in HW_{0}^{1,2}\left(
B_{4R}\right)  $, so $u_{0}$\ is a weak solution to (\ref{3.11}).

Now, let $v_{k}$ be the unique solution to the Dirichlet problem%
\begin{equation}
\left\{
\begin{array}
[c]{l}%
X_{i}\left(  \left(  a_{\alpha\beta}^{ijk}\right)  _{B_{4R}}X_{j}v_{k}\right)
=0\text{ \ in }B_{4R}\\
v_{k}-u_{0}\in HW_{0}^{1,2}\left(  B_{4R}\right)
\end{array}
\right.  \label{aux system}%
\end{equation}
(see Remark \ref{existence}). By (\ref{elliptic condition}) and using
$v_{k}-u_{0}$ as a test function in the definition of solution to
(\ref{aux system}) we have%
\begin{align*}
\mu\int_{B_{4R}}\left\vert Xv_{k}-Xu_{0}\right\vert ^{2}dx  &  \leq
\int_{B_{4R}}\left(  a_{\alpha\beta}^{ijk}\right)  _{B_{4R}}\left(  X_{j}%
v_{k}^{\beta}-X_{j}u_{0}^{\beta}\right)  \left(  X_{i}v_{k}^{\alpha}%
-X_{i}u_{0}^{\alpha}\right)  dx\\
&  =-\int_{B_{4R}}\left(  a_{\alpha\beta}^{ijk}\right)  _{B_{4R}}X_{j}%
u_{0}^{\beta}\left(  X_{i}v_{k}^{\alpha}-X_{i}u_{0}^{\alpha}\right)  dx,
\end{align*}
since $u_{0}$ is a weak solution to (\ref{3.11})
\begin{align*}
&  =\int_{B_{4R}}\left(  \bar{a}_{\alpha\beta}^{ij}-\left(  a_{\alpha\beta
}^{ijk}\right)  _{B_{4R}}\right)  X_{j}u_{0}^{\beta}\left(  X_{i}v_{k}%
^{\alpha}-X_{i}u_{0}^{\alpha}\right)  dx\\
&  \leq\left\vert \left(  \bar{a}_{\alpha\beta}^{ij}-\left(  a_{\alpha\beta
}^{ijk}\right)  _{B_{4R}}\right)  \right\vert \int_{B_{4R}}\left\vert
X_{j}u_{0}^{\beta}\right\vert \left\vert X_{i}v_{k}^{\alpha}-X_{i}%
u_{0}^{\alpha}\right\vert dx\\
&  \leq c\left(  N\right)  \max_{i,j,\alpha,\beta}\left\vert \left(  \bar
{a}_{\alpha\beta}^{ij}-\left(  a_{\alpha\beta}^{ijk}\right)  _{B_{4R}}\right)
\right\vert \left(  \int_{B_{4R}}\left\vert Xu_{0}\right\vert ^{2}dx\right)
^{1/2}\cdot\\
&  \cdot\left(  \int_{B_{4R}}\left\vert Xv_{k}-Xu_{0}\right\vert
^{2}dx\right)  ^{1/2},
\end{align*}
which implies%
\begin{equation}
\mu\left(  \int_{B_{4R}}\left\vert Xv_{k}-Xu_{0}\right\vert ^{2}dx\right)
^{1/2}\leq c\max_{i,j,\alpha,\beta}\left\vert \left(  \bar{a}_{\alpha\beta
}^{ij}-\left(  a_{\alpha\beta}^{ijk}\right)  _{B_{4R}}\right)  \right\vert
\left(  \int_{B_{4R}}\left\vert Xu_{0}\right\vert ^{2}dx\right)  ^{1/2}.
\label{3.9b}%
\end{equation}
Inequalities (\ref{3.9}) and (\ref{3.9b}) imply
\[
\left\Vert Xv_{k}-Xu_{0}\right\Vert _{L^{2}\left(  B_{4R}\right)  }%
\rightarrow0\text{ as }k\rightarrow0.
\]
This convergence, the fact that $v_{k}-u_{0}\in HW_{0}^{1,2}\left(
B_{4R}\right)  $ and (\ref{poincare inequality 1}) imply%
\begin{equation}
\left\Vert v_{k}-u_{0}\right\Vert _{L^{2}\left(  B_{4R}\right)  }%
\rightarrow0\text{ as }k\rightarrow0. \label{v_k - u0}%
\end{equation}
By (\ref{3.8}) and (\ref{v_k - u0}) we can write:%
\begin{align}
&  \left\Vert v_{k}-\left(  u_{k}-\left(  u_{k}\right)  _{B_{4R}}\right)
\right\Vert _{L^{2}\left(  B_{4R}\right)  }\nonumber\\
&  \leq\left\Vert u_{0}-\left(  u_{k}-\left(  u_{k}\right)  _{B_{4R}}\right)
\right\Vert _{L^{2}\left(  B_{4R}\right)  }+\left\Vert v_{k}-u_{0}\right\Vert
_{L^{2}\left(  B_{4R}\right)  }\rightarrow0. \label{vanish}%
\end{align}
On the other hand, $v_{k}+\left(  u_{k}\right)  _{B_{4R}}$ is still a weak
solution to (\ref{system const}), hence (\ref{3.6}) implies%
\[
\frac{1}{R^{2}}\frac{1}{\left\vert B_{4R}\right\vert }\int_{B_{4R}}\left\vert
v_{k}-\left(  u_{k}-\left(  u_{k}\right)  _{B_{4R}}\right)  \right\vert
^{2}dx>\varepsilon_{0}^{2},
\]
which contradicts (\ref{vanish}). So we have proved the assertion, for some
$\delta$ possibly depending on $\varepsilon,R,\mu$.

Let us now fix a particular $R_{0},$ and let $R$ be any number $\leq R_{0}$.
Assume $u\,\ $is a weak solution$\ $to system (\ref{system}) in $B_{4R}%
\Subset\Omega$ satisfying (\ref{3.1}). Just to simplify notations, assume that
the center of $B_{4R}$ is the origin, and define:%
\begin{align*}
\widetilde{u}\left(  x\right)   &  =\frac{R_{0}}{R}u\left(  D\left(  \frac
{R}{R_{0}}\right)  x\right)  ;\\
\widetilde{a}_{ij}^{\alpha\beta}\left(  x\right)   &  =a_{ij}^{\alpha\beta
}\left(  D\left(  \frac{R}{R_{0}}\right)  x\right)  ;\\
\widetilde{f_{\alpha}^{i}}\left(  x\right)   &  =f_{\alpha}^{i}\left(
D\left(  \frac{R}{R_{0}}\right)  x\right)  .
\end{align*}
Then, one can check that the function $\widetilde{u}$ solves the system
\[
X_{i}\left(  \widetilde{a}_{\alpha\beta}^{ij}(x)X_{j}\widetilde{u}^{\beta
}\right)  =X_{i}\widetilde{f_{\alpha}^{i}}\text{ in }B_{4R_{0}}.
\]
To see this, for any $\phi\in C_{0}^{\infty}\left(  B_{4R}\right)  $, let
$\widetilde{\phi}\left(  x\right)  =\frac{R_{0}}{R}\phi\left(  D\left(
\frac{R}{R_{0}}\right)  x\right)  $; then $\widetilde{\phi}\in C_{0}^{\infty
}\left(  B_{4R_{0}}\right)  $ and%
\begin{align*}
&  \int_{B_{4R_{0}}}\widetilde{a}_{ij}^{\alpha\beta}\left(  x\right)
X_{j}\widetilde{u}^{\beta}\left(  x\right)  X_{i}\widetilde{\phi}^{\alpha
}\left(  x\right)  dx\\
&  =\int_{B_{4R_{0}}}a_{ij}^{\alpha\beta}\left(  D\left(  \frac{R}{R_{0}%
}\right)  x\right)  \left(  X_{j}u^{\beta}\right)  \left(  D\left(  \frac
{R}{R_{0}}\right)  x\right)  \left(  X_{i}\phi^{\alpha}\right)  \left(
D\left(  \frac{R}{R_{0}}\right)  x\right)  dx\\
&  =\left(  \frac{R_{0}}{R}\right)  ^{Q}\int_{B_{4R}}a_{ij}^{\alpha\beta
}\left(  y\right)  \left(  X_{j}u^{\beta}\right)  \left(  y\right)  \left(
X_{i}\phi^{\alpha}\right)  \left(  y\right)  dy\\
&  =\left(  \frac{R_{0}}{R}\right)  ^{Q}\int_{B_{4R}}f_{\alpha}^{i}\left(
y\right)  \left(  X_{i}\phi^{\alpha}\right)  \left(  y\right)  dy\\
&  =\int_{B_{4R_{0}}}f_{\alpha}^{i}\left(  D\left(  \frac{R}{R_{0}}\right)
x\right)  \left(  X_{i}\phi^{\alpha}\right)  \left(  D\left(  \frac{R}{R_{0}%
}\right)  x\right)  dx\\
&  =\int_{B_{4R_{0}}}\widetilde{f}_{i}^{\alpha}\left(  x\right)
X_{i}\widetilde{\phi}^{\alpha}\left(  x\right)  dx.
\end{align*}
Also, note that the $\widetilde{a}_{\alpha\beta}^{ij}$'s satisfy condition
(\ref{elliptic condition}) with the same $\mu$. Let $\delta=\delta\left(
\varepsilon,R_{0},\mu\right)  $ be the number found in the first part of the
proof, and assume that $u,\mathbf{F},a_{ij}^{\alpha\beta}$ satisfy (\ref{3.1})
on $B_{4R}$ for this $\delta$; then $\widetilde{u},\widetilde{\mathbf{F}%
},\widetilde{a}_{ij}^{\alpha\beta}$ satisfy (\ref{3.1}) on $B_{4R_{0}}$ for
the same $\delta$:%
\begin{align*}
\frac{1}{\left\vert B_{4R_{0}}\right\vert }\int_{B_{4R_{0}}}\left\vert
X\widetilde{u}\left(  x\right)  \right\vert ^{2}dx  &  =\frac{1}{\left\vert
B_{4R_{0}}\right\vert }\int_{B_{4R_{0}}}\left\vert \left(  Xu\right)  \left(
D\left(  \frac{R}{R_{0}}\right)  x\right)  \right\vert ^{2}dx\\
=\frac{1}{\left\vert B_{4R_{0}}\right\vert }\left(  \frac{R_{0}}{R}\right)
^{Q}\int_{B_{4R}}\left\vert \left(  Xu\right)  \left(  y\right)  \right\vert
^{2}dy  &  =\frac{1}{\left\vert B_{4R}\right\vert }\int_{B_{4R}}\left\vert
\left(  Xu\right)  \left(  y\right)  \right\vert ^{2}dy\leq1;
\end{align*}%
\begin{align*}
&  \frac{1}{\left\vert B_{4R_{0}}\right\vert }\int_{B_{4R_{0}}}\left(
\left\vert \widetilde{\mathbf{F}}\right\vert ^{2}+\left\vert \widetilde
{a}_{\alpha\beta}^{ij}-\left(  \widetilde{a}_{\alpha\beta}^{ij}\right)
_{B_{4R_{0}}}\right\vert ^{2}\right)  dx\\
&  =\frac{1}{\left\vert B_{4R}\right\vert }\int_{B_{4R}}\left(  \left\vert
\mathbf{F}\right\vert ^{2}+\left\vert a_{\alpha\beta}^{ij}-\left(
a_{\alpha\beta}^{ij}\right)  _{B_{4R}}\right\vert ^{2}\right)  dx\leq\delta.
\end{align*}

Hence, by the first part of the proof, there exists a weak solution
$\widetilde{v}$ to the following homogeneous system with constant
coefficients:
\[
X_{i}\left(  \left(  \widetilde{a}_{\alpha\beta}^{ij}\right)  _{B_{4R_{0}}%
}X_{j}\widetilde{v}^{\beta}(x)\right)  =0\text{ in }B_{4R_{0}}%
\]
such that%
\[
\frac{1}{R_{0}^{2}}\frac{1}{\left\vert B_{4R_{0}}\right\vert }\int_{B_{4R_{0}%
}}\left\vert \widetilde{u}-\widetilde{v}\right\vert ^{2}dx\leq\varepsilon
^{2}.
\]
Then, the function%
\[
v\left(  x\right)  =\frac{R}{R_{0}}\widetilde{v}\left(  D\left(  \frac{R_{0}%
}{R}\right)  x\right)
\]
satisfies
\[
X_{i}\left(  \left(  a_{\alpha\beta}^{ij}\right)  _{B_{4R}}X_{j}v^{\beta
}(x)\right)  =0\text{ in }B_{4R}%
\]
and%
\begin{align*}
&  \frac{1}{R^{2}}\frac{1}{\left\vert B_{4R}\right\vert }\int_{B_{4R}%
}\left\vert u\left(  x\right)  -v\left(  x\right)  \right\vert ^{2}dx\\
&  =\frac{1}{R^{2}}\frac{1}{\left\vert B_{4R}\right\vert }\int_{B_{4R}%
}\left\vert \frac{R}{R_{0}}\widetilde{u}\left(  D\left(  \frac{R_{0}}%
{R}\right)  x\right)  -\frac{R}{R_{0}}\widetilde{v}\left(  D\left(
\frac{R_{0}}{R}\right)  x\right)  \right\vert ^{2}dx\\
&  =\frac{1}{R_{0}^{2}}\frac{1}{\left\vert B_{4R}\right\vert }\int_{B_{4R}%
}\left\vert \widetilde{u}\left(  D\left(  \frac{R_{0}}{R}\right)  x\right)
-\widetilde{v}\left(  D\left(  \frac{R_{0}}{R}\right)  x\right)  \right\vert
^{2}dx\\
&  =\frac{1}{R_{0}^{2}}\frac{1}{\left\vert B_{4R_{0}}\right\vert }%
\int_{B_{4R_{0}}}\left\vert \widetilde{u}-\widetilde{v}\right\vert ^{2}%
dx\leq\varepsilon^{2}.
\end{align*}
We have therefore proved that the assertion holds with $\delta$ depending on
$R_{0}$ but independent of $R\leq R_{0}$.
\end{proof}

The following technical lemma is adapted from \cite[lemma 4.1, p.27]{c.w}.

\begin{lemma}
\label{cw-lemma}Let $\psi(t)$ be a bounded nonnegative function defined on the
interval $\left[  T_{0},T_{1}\right]  $, where $T_{1}>T_{0}\geq0$. Suppose
that for any $T_{0}\leq t\leq s\leq T_{1}$, $\psi$ satisfies
\[
\psi(t)\leq\vartheta\psi(s)+\frac{A}{\left(  s-t\right)  ^{\beta}}+B,
\]
where $\vartheta$, $A$, $B$, $\beta$ are nonnegative constants, and
$\vartheta<\frac{1}{3}$. Then
\[
\psi(\rho)\leq c_{\beta}\left[  \frac{A}{\left(  R-\rho\right)  ^{\beta}%
}+B\right]  ,\forall\rho,T_{0}\leq\rho<R\leq T_{1},
\]
where $c_{\beta}$ only depends on $\beta$.
\end{lemma}

We are going to enforce the previous theorem with the following

\begin{theorem}
\label{approximate gradient}For any $\varepsilon>0$ $R_{0}>0,$ there is a
small $\delta=\delta\left(  \varepsilon,R_{0},\mu\right)  >0$ such that for
any $R\leq R_{0}$, if $u\,\ $is a weak solution$\ $of system (\ref{system}) in
$B_{4R}\Subset\Omega$ and (\ref{3.1}) holds, then there exists a weak solution
$v$ to (\ref{3.2}) such that%
\[
\frac{1}{\left\vert B_{2R}\right\vert }\int_{B_{2R}}\left\vert
Xu-Xv\right\vert ^{2}dx\leq\varepsilon^{2}.
\]

\end{theorem}

\begin{proof}
By Theorem \ref{aproximate for u}, we know that for any $\eta>0$, there exist
a small $\delta=\delta\left(  \eta,R_{0},\mu\right)  >0$ and a weak solution
$v$ of (\ref{3.2}) in $B_{4R}$, such that%
\begin{equation}
\frac{1}{R^{2}}\frac{1}{\left\vert B_{4R}\right\vert }\int_{B_{4R}}\left\vert
u-v\right\vert ^{2}dx\leq\eta^{2}, \label{3.16(1)}%
\end{equation}
provided (\ref{3.1}) holds.

Let us note that $u-v$ is a weak solution to the system%
\begin{equation}
X_{i}\left(  a_{\alpha\beta}^{ij}\left(  x\right)  X_{j}\left(  u^{\beta
}-v^{\beta}\right)  (x)\right)  =X_{i}\left(  f_{\alpha}^{i}\left(  x\right)
-\left(  a_{\alpha\beta}^{ij}\left(  x\right)  -\left(  a_{\alpha\beta}%
^{ij}\right)  _{B_{4R}}\right)  X_{j}v^{\beta}\right)  \text{ } \label{3.17}%
\end{equation}
in $B_{4R}$. For any $2R\leq s<t\leq3R$, we choose a cutoff function
$\varphi\left(  x\right)  $ which satisfies
\[
0<\varphi\left(  x\right)  \leq1\text{ in }B_{3R}\text{, }\varphi\left(
x\right)  \equiv1\text{ in }B_{s}\text{, }\varphi\left(  x\right)
\equiv0\text{ in }B_{3R}\backslash B_{t}\text{ }%
\]
and
\[
\left\vert X\varphi\left(  x\right)  \right\vert \leq\frac{c}{t-s}\text{ in
}B_{4R}\text{.}%
\]
Taking $\left(  u-v\right)  \varphi$ as a test function, it follows by
(\ref{3.17}) that
\begin{align*}
&  \mu\int_{B_{s}}\left\vert X\left(  u-v\right)  \right\vert ^{2}dx\\
&  \leq\int_{B_{t}}\varphi\left(  x\right)  a_{\alpha\beta}^{ij}\left(
x\right)  X_{j}\left(  u^{\beta}-v^{\beta}\right)  X_{i}\left(  u^{\alpha
}-v^{\alpha}\right)  dx\\
&  =\int_{B_{t}}\left(  f_{\alpha}^{i}\left(  x\right)  -\left(
a_{\alpha\beta}^{ij}\left(  x\right)  -\left(  a_{\alpha\beta}^{ij}\right)
_{B_{4R}}\right)  X_{j}v^{\beta}\right)  X_{i}\left(  \left(  u^{\alpha
}-v^{\alpha}\right)  \varphi\right)  dx\\
&  -\int_{B_{t}}a_{\alpha\beta}^{ij}\left(  x\right)  \left(  u^{\alpha
}-v^{\alpha}\right)  X_{j}\left(  u^{\beta}-v^{\beta}\right)  X_{i}\varphi dx.
\end{align*}
By the properties of $\varphi$, Young's inequality and
(\ref{elliptic condition}),%
\begin{align*}
&  \int_{B_{s}}\left\vert Xu-Xv\right\vert ^{2}dx\\
&  \leq c\int_{B_{t}}\left(  \left\vert \mathbf{F}\right\vert +\max
_{i,j,\alpha,\beta}\left\vert a_{\alpha\beta}^{ij}\left(  x\right)  -\left(
a_{\alpha\beta}^{ij}\left(  x\right)  \right)  _{B_{4R}}\right\vert \left\vert
Xv\right\vert \right)  ^{2}dx+\\
&  +\frac{1}{4}\int_{B_{t}}\left\vert Xu-Xv\right\vert ^{2}dx+\frac{c}{\left(
t-s\right)  ^{2}}\int_{B_{t}}\left\vert u-v\right\vert ^{2}dx
\end{align*}%
\begin{align*}
&  \leq c\int_{4R}\left\vert \mathbf{F}\right\vert ^{2}dx+\sup_{B_{3R}%
}\left\vert Xv\right\vert ^{2}\cdot\max_{i,j,\alpha,\beta}\int_{B_{4R}%
}\left\vert a_{\alpha\beta}^{ij}\left(  x\right)  -\left(  a_{\alpha\beta
}^{ij}\left(  x\right)  \right)  _{B_{4R}}\right\vert ^{2}dx\\
&  +\frac{c}{\left(  t-s\right)  ^{2}}\int_{B_{4R}}\left\vert u-v\right\vert
^{2}dx+\frac{1}{4}\int_{B_{t}}\left\vert Xu-Xv\right\vert ^{2}dx.
\end{align*}
Setting%
\begin{align*}
\psi\left(  s\right)   &  =\int_{B_{s}}\left\vert Xu-Xv\right\vert ^{2}dx,\\
B  &  =c\int_{B_{4R}}\left\vert \mathbf{F}\right\vert ^{2}dx+\sup_{B_{3R}%
}\left\vert Xv\right\vert ^{2}\cdot\max_{i,j,\alpha,\beta}\int_{B_{4R}%
}\left\vert a_{\alpha\beta}^{ij}\left(  x\right)  -\left(  a_{\alpha\beta
}^{ij}\left(  x\right)  \right)  _{B_{4R}}\right\vert ^{2}dx,\\
A  &  =\int_{B_{4R}}\left\vert u-v\right\vert ^{2}dx,\beta=2,
\end{align*}
by Lemma \ref{cw-lemma} we deduce%
\begin{align}
&  \int_{B_{2R}}\left\vert Xu-Xv\right\vert ^{2}dx\leq\frac{c}{R^{2}}%
\int_{B_{4R}}\left\vert u-v\right\vert ^{2}dx+c\int_{B_{4R}}\left\vert
\mathbf{F}\right\vert ^{2}dx\nonumber\\
&  +c\underset{B_{3R}}{\sup}\left\vert Xv\right\vert ^{2}\cdot\max
_{i,j,\alpha,\beta}\int_{B_{4R}}\left\vert \left(  a_{\alpha\beta}^{ij}\left(
x\right)  -\left(  a_{\alpha\beta}^{ij}\left(  x\right)  \right)  _{B_{4R}%
}\right)  \right\vert ^{2}dx. \label{w(1,1)b}%
\end{align}
By Theorem \ref{L2.4}, since $v-u_{B_{4R}}$ is still a solution to the system
(\ref{3.2}) in $B_{4R}$ we can write:%
\begin{align*}
\underset{B_{3R}}{\sup}\left\vert Xv\right\vert  &  \leq\frac{c}{R}\left\vert
B_{R}\right\vert ^{-1/2}\left\Vert v-u_{B_{4R}}\right\Vert _{L^{2}\left(
B_{4R}\right)  }\\
&  \leq\frac{c}{R}\left\vert B_{R}\right\vert ^{-1/2}\left(  \left\Vert
u-v\right\Vert _{L^{2}\left(  B_{4R}\right)  }+\left\Vert u-u_{B_{4R}%
}\right\Vert _{L^{2}\left(  B_{4R}\right)  }\right)
\end{align*}
by (\ref{3.16(1)}), (\ref{poincare inequality}) and assumption (\ref{3.1}) on
$u$
\begin{equation}
\leq c\eta+c\left\vert B_{R}\right\vert ^{-1/2}\left\Vert Xu\right\Vert
_{L^{2}\left(  B_{4R}\right)  }\leq c\left(  \eta+1\right)  \leq N_{0},
\label{w(1,1)}%
\end{equation}
for some absolute constant $N_{0}$ when $\eta$ is, say, any number $\leq1$.

By (\ref{w(1,1)b}) and (\ref{w(1,1)}) we have%
\begin{align*}
&  \frac{1}{\left\vert B_{2R}\right\vert }\int_{B_{2R}}\left\vert
Xu-Xv\right\vert ^{2}dx\\
&  \leq\frac{c}{\left\vert B_{4R}\right\vert }\int_{B_{4R}}\left\vert
\mathbf{F}\right\vert ^{2}dx+\frac{cN_{0}}{\left\vert B_{4R}\right\vert }%
\max_{i,j,\alpha,\beta}\int_{B_{4R}}\left\vert a_{\alpha\beta}^{ij}\left(
x\right)  -\left(  a_{\alpha\beta}^{ij}\right)  _{B_{4R}}\right\vert ^{2}dx\\
&  +\frac{c}{R^{2}}\frac{1}{\left\vert B_{4R}\right\vert }\int_{B_{4R}%
}\left\vert u-v\right\vert ^{2}dx,
\end{align*}
by (\ref{3.16(1)}) and (\ref{3.1})%
\begin{align*}
&  \leq\frac{c}{\left\vert B_{4R}\right\vert }\int_{B_{4R}}\left(  \left\vert
\mathbf{F}\right\vert ^{2}+\max_{i,j,\alpha,\beta}\left\vert a_{\alpha\beta
}^{ij}\left(  x\right)  -\left(  a_{\alpha\beta}^{ij}\left(  x\right)
\right)  _{B_{4R}}\right\vert ^{2}\right)  dx+c\eta^{2}\\
&  \leq c\left(  \delta^{2}+\eta^{2}\right)  <\varepsilon^{2},
\end{align*}
for a suitable choice of $\eta,$ and after possibly diminishing $\delta.$ This
ends the proof.
\end{proof}

\section{Estimates on the maximal function of $\left\vert Xu\right\vert ^{2}%
$\label{sec maximal}}

\begin{definition}
Let $B_{R}\Subset\Omega$. For every $f\in L^{1}\left(  B_{R}\right)  $, define
the Hardy--Littlewood maximal function of $f$ by%
\[
\mathcal{M}_{B_{R}}\left(  f\right)  \left(  x\right)  =\sup_{r>0}\frac
{1}{\left\vert B_{r}\left(  x\right)  \cap B_{R}\right\vert }\int
_{B_{r}\left(  x\right)  \cap B_{R}}\left\vert f\left(  y\right)  \right\vert
dy.
\]

\end{definition}

Since $\left(  B_{R},d_{X},dx\right)  $ is a space of homogeneous type (see
Remark \ref{d-regular and homegeneous}), by \cite[Thm.2.1 p.71]{cw} the
following holds:

\begin{lemma}
\label{maximal function}Let $f\in L^{1}\left(  B_{R}\right)  $, then

(i) $\mathcal{M}_{B_{R}}\left(  f\right)  \left(  x\right)  $ is finite almost
everywhere in $B_{R}$;

(ii) for every $\alpha>0$,
\[
\left\vert \left\{  x\in B_{R}:\mathcal{M}_{B_{R}}\left(  f\right)  \left(
x\right)  >\alpha\right\}  \right\vert \leq\frac{c_{1}}{\alpha}\int_{B_{R}%
}\left\vert f\left(  y\right)  \right\vert dy;
\]

(iii) if $f\in L^{p}\left(  B_{R}\right)  $ with $1<p<\infty$, then
$\mathcal{M}_{B_{R}}\left(  f\right)  \in L^{p}\left(  B_{R}\right)  $ and
\[
\left\Vert \mathcal{M}_{B_{R}}\left(  f\right)  \right\Vert _{L^{p}\left(
B_{R}\right)  }\leq c_{p}\left\Vert f\right\Vert _{L^{p}\left(  B_{R}\right)
},
\]
where the constants $c_{p}$ only depend on $p$ and $\mathbb{G}$ (but are
independent of $B_{R}$).
\end{lemma}

The last statement about the dependence of the constants requires some
explanation. In any space of homogeneous type these constants depend on the
two constants of the space, namely the one appearing in the \textquotedblleft
quasitriangle inequality\textquotedblright\ (\ref{quasitriangle}) and the
doubling constant appearing in (\ref{doubling}). In our case the first
constant is $1$ (since $d_{X}$ is a distance) and the second is independent of
$R,$ by Remark \ref{d-regular and homegeneous}. Hence $c_{p}$ is independent
of $R$.

\begin{theorem}
\label{main lemma}There exists an absolute constant $N_{1}$ such that for any
$\varepsilon>0$, $R_{0}>0,$ there is a small $\delta=\delta\left(
\varepsilon,R_{0},\mu\right)  >0$ such that for any $R\leq R_{0}/2,$ $z\in
B_{R}\left(  \overline{x}\right)  \subset B_{11R}\left(  \overline{x}\right)
\Subset\Omega$ and $0<r\leq2R$, if $u$ is a weak solution of (\ref{system}) in
$B_{11R}\left(  \overline{x}\right)  $ with
\begin{align}
B_{r}\left(  z\right)  \cap\left\{  x\in B_{R}\left(  \overline{x}\right)
:\mathcal{M}_{B_{11R}\left(  \overline{x}\right)  }\left(  \left\vert
Xu\right\vert ^{2}\right)  \left(  x\right)  \leq1\right\}   &  \cap
\nonumber\\
\cap\left\{  x\in B_{R}\left(  \overline{x}\right)  :\mathcal{M}%
_{B_{11R}\left(  \overline{x}\right)  }\left(  \left\vert \mathbf{F}%
\right\vert ^{2}\right)  \left(  x\right)  \leq\delta^{2}\right\}  \neq &
\emptyset\label{3.18}%
\end{align}
and the coefficients $a_{\alpha\beta}^{ij}\left(  x\right)  $ are $\left(
\delta,4r\right)  $-vanishing in $B_{R}\left(  \overline{x}\right)  $, then%
\begin{equation}
\left\vert B_{r}\left(  z\right)  \cap\left\{  x\in B_{R}\left(  \overline
{x}\right)  :\mathcal{M}_{B_{11R}\left(  \overline{x}\right)  }\left(
\left\vert Xu\right\vert ^{2}\right)  \left(  x\right)  >N_{1}^{2}\right\}
\right\vert <\varepsilon\left\vert B_{r}\left(  z\right)  \right\vert .
\label{3.19}%
\end{equation}

\end{theorem}

\begin{proof}
Fix $\varepsilon,R_{0}>0;$ the number $\delta$ will be chosen later. By
(\ref{3.18}), there exists a point $x_{0}\in B_{r}\left(  z\right)  $, such
that for any $\rho>0$,%
\begin{align}
\frac{1}{\left\vert B_{\rho}\left(  x_{0}\right)  \cap B_{11R}\left(
\overline{x}\right)  \right\vert }\int_{B_{\rho}\left(  x_{0}\right)  \cap
B_{11R}\left(  \overline{x}\right)  }\left\vert Xu\right\vert ^{2}dx  &
\leq1,\label{add3}\\
\frac{1}{\left\vert B_{\rho}\left(  x_{0}\right)  \cap B_{11R}\left(
\overline{x}\right)  \right\vert }\int_{B_{\rho}\left(  x_{0}\right)  \cap
B_{11R}\left(  \overline{x}\right)  }\left\vert \mathbf{F}\right\vert ^{2}dx
&  \leq\delta^{2}. \label{3.20}%
\end{align}
Since $z,x_{0}\in B_{R}\left(  \overline{x}\right)  $ and $r\leq2R$, we have
the inclusions: $B_{4r}\left(  z\right)  \subset B_{5r}\left(  x_{0}\right)
\subset B_{11R}\left(  \overline{x}\right)  $ and $B_{5r}\left(  x_{0}\right)
\subset B_{6r}\left(  z\right)  $. Then by (\ref{3.20}) with $\rho=5r$ we have
that%
\begin{equation}
\frac{1}{\left\vert B_{4r}\left(  z\right)  \right\vert }\int_{B_{4r}\left(
z\right)  }\left\vert \mathbf{F}\right\vert ^{2}dx\leq\frac{\left\vert
B_{6r}\left(  z\right)  \right\vert }{\left\vert B_{4r}\left(  z\right)
\right\vert }\frac{1}{\left\vert B_{5r}\left(  x_{0}\right)  \right\vert }%
\int_{B_{5r}\left(  x_{0}\right)  }\left\vert \mathbf{F}\right\vert ^{2}%
dx\leq\left(  \frac{6}{4}\right)  ^{Q}\delta^{2}. \label{3.21}%
\end{equation}
Similarly, by (\ref{add3}) we find%
\begin{equation}
\frac{1}{\left\vert B_{4r}\left(  z\right)  \right\vert }\int_{B_{4r}\left(
z\right)  }\left\vert Xu\right\vert ^{2}dx\leq\left(  \frac{6}{4}\right)
^{Q}. \label{3.22}%
\end{equation}
By (\ref{3.21}), (\ref{3.22}) and the assumption on $a_{\alpha\beta}%
^{ij}\left(  x\right)  $, we can apply Theorem \ref{approximate gradient}
(with $u$ replaced by $\left(  \frac{4}{6}\right)  ^{Q}u$ and $\mathbf{F}$
replaced by $\left(  \frac{4}{6}\right)  ^{Q}\mathbf{F}$) on the ball
$B_{4r}\left(  z\right)  $ (recall that $r\leq R_{0}$) and obtain that for any
$\eta>0$, there exists a small $\delta=\delta\left(  \eta,R_{0},\mu\right)  $
and a weak solution $v$ to
\[
X_{i}\left(  \left(  a_{\alpha\beta}^{ij}\right)  _{B_{4r}\left(  z\right)
}X_{j}v\right)  =0\text{ in }B_{4r}\left(  z\right)
\]
such that%
\begin{equation}
\frac{1}{\left\vert B_{2r}\left(  z\right)  \right\vert }\int_{B_{2r}\left(
z\right)  }\left\vert X\left(  u-v\right)  \right\vert ^{2}dx\leq\eta^{2}.
\label{3.23}%
\end{equation}

Also, recall the interior $HW^{1,\infty}$ regularity of $v$ (\ref{w(1,1)}): \
\begin{equation}
\left\Vert Xv\right\Vert _{L^{\infty}\left(  B_{3r}\left(  z\right)  \right)
}^{2}\leq N_{0}^{2}. \label{3.24}%
\end{equation}
Now, pick%
\begin{equation}
N_{1}^{2}=\max\left\{  \frac{5^{Q}}{c_{d}},4N_{0}^{2}\right\}  . \label{N1}%
\end{equation}
Then we claim that%
\begin{align}
&  \left\{  x\in B_{R}\left(  \overline{x}\right)  :\mathcal{M}_{B_{11R}%
\left(  \overline{x}\right)  }\left(  \left\vert Xu\right\vert ^{2}\right)
\left(  x\right)  >N_{1}^{2}\right\}  \cap B_{r}\left(  z\right) \nonumber\\
&  \subset\left\{  x\in B_{R}\left(  \overline{x}\right)  :\mathcal{M}%
_{B_{2r}\left(  z\right)  }\left(  \left\vert X\left(  u-v\right)  \right\vert
^{2}\right)  \left(  x\right)  >N_{0}^{2}\right\}  \cap B_{r}\left(  z\right)
. \label{3.25}%
\end{align}

To see this, suppose
\begin{equation}
x_{1}\in\left\{  x\in B_{R}\left(  \overline{x}\right)  \cap B_{r}\left(
z\right)  :\mathcal{M}_{B_{2r}\left(  z\right)  }\left(  \left\vert X\left(
u-v\right)  \right\vert ^{2}\right)  \left(  x\right)  \leq N_{0}^{2}\right\}
.\label{3.26}%
\end{equation}
When $\rho\leq r,$ it follows $B_{\rho}\left(  x_{1}\right)  \subset
B_{2r}\left(  z\right)  \subset B_{5R}\left(  \overline{x}\right)  $, then
(\ref{3.26}) and (\ref{3.24}) imply%
\begin{align}
&  \frac{1}{\left\vert B_{\rho}\left(  x_{1}\right)  \cap B_{11R}\left(
\overline{x}\right)  \right\vert }\int_{B_{\rho}\left(  x_{1}\right)  \cap
B_{11R}\left(  \overline{x}\right)  }\left\vert Xu\right\vert ^{2}dx=\frac
{1}{\left\vert B_{\rho}\left(  x_{1}\right)  \right\vert }\int_{B_{\rho
}\left(  x_{1}\right)  }\left\vert Xu\right\vert ^{2}dx\label{3.27}\\
&  \leq\frac{2}{\left\vert B_{\rho}\left(  x_{1}\right)  \right\vert }%
\int_{B_{\rho}\left(  x_{1}\right)  }\left(  \left\vert X\left(  u-v\right)
\right\vert ^{2}+\left\vert Xv\right\vert ^{2}\right)  dx\leq4N_{0}^{2}\leq
N_{1}^{2}.\nonumber
\end{align}
When $\rho>r$, since $x_{1},x_{0}\in B_{r}\left(  z\right)  $ we have
$d\left(  x_{1},x_{0}\right)  <2r<2\rho$; it follows $B_{\rho}\left(
x_{1}\right)  \subset B_{3\rho}\left(  x_{0}\right)  \subset B_{5\rho}\left(
x_{1}\right)  $. Then by Remark \ref{d-regular and homegeneous} and
(\ref{add3}) we have
\begin{align}
&  \frac{1}{\left\vert B_{\rho}\left(  x_{1}\right)  \cap B_{11R}\left(
\overline{x}\right)  \right\vert }\int_{B_{\rho}\left(  x_{1}\right)  \cap
B_{11R}\left(  \overline{x}\right)  }\left\vert Xu\right\vert ^{2}%
dx\nonumber\\
&  \leq\frac{1}{c_{d}\left\vert B_{\rho}\left(  x_{1}\right)  \right\vert
}\int_{B_{3\rho}\left(  x_{0}\right)  \cap B_{11R}\left(  \overline{x}\right)
}\left\vert Xu\right\vert ^{2}dx\nonumber\\
&  =\frac{5^{Q}}{c_{d}\left\vert B_{5\rho}\left(  x_{1}\right)  \right\vert
}\int_{B_{3\rho}\left(  x_{0}\right)  \cap B_{11R}\left(  \overline{x}\right)
}\left\vert Xu\right\vert ^{2}dx\nonumber\\
&  \leq\frac{5^{Q}}{c_{d}\left\vert B_{3\rho}\left(  x_{0}\right)  \cap
B_{11R}\left(  \overline{x}\right)  \right\vert }\int_{B_{3\rho}\left(
x_{0}\right)  \cap B_{11R}\left(  \overline{x}\right)  }\left\vert
Xu\right\vert ^{2}dx\nonumber\\
&  \leq\frac{5^{Q}}{c_{d}}\leq N_{1}^{2}.\label{3.28}%
\end{align}
By (\ref{3.27}) and (\ref{3.28}), we have%
\begin{equation}
x_{1}\in\left\{  x\in B_{R}\left(  \overline{x}\right)  :\mathcal{M}%
_{B_{11R}\left(  \overline{x}\right)  }\left(  \left\vert Xu\right\vert
^{2}\right)  \leq N_{1}^{2}\right\}  \cap B_{r}\left(  z\right)  .\label{3.29}%
\end{equation}
Thus, inclusion (\ref{3.25}) follows from the fact that (\ref{3.26}) implies
(\ref{3.29}).

By (\ref{3.25}), Lemma \ref{maximal function} (ii) and (\ref{3.23}) , we have%
\begin{align*}
&  \left\vert \left\{  x\in B_{R}\left(  \overline{x}\right)  :\mathcal{M}%
_{B_{11R}\left(  \overline{x}\right)  }\left(  \left\vert Xu\right\vert
^{2}\right)  \left(  x\right)  >N_{1}^{2}\right\}  \cap B_{r}\left(  z\right)
\right\vert \\
&  \leq\left\vert \left\{  x\in B_{2r}\left(  z\right)  :\mathcal{M}%
_{B_{2r}\left(  z\right)  }\left(  \left\vert X\left(  u-v\right)  \right\vert
^{2}\right)  \left(  x\right)  >N_{0}^{2}\right\}  \right\vert \\
&  \leq\frac{c}{N_{0}^{2}}\int_{B_{2r}\left(  z\right)  }\left\vert X\left(
u-v\right)  \right\vert ^{2}dx\\
&  \leq c\eta^{2}\left\vert B_{2r}\left(  z\right)  \right\vert =c2^{Q}%
\eta^{2}\left\vert B_{r}\left(  z\right)  \right\vert \\
&  =\varepsilon^{2}\left\vert B_{r}\left(  z\right)  \right\vert .
\end{align*}
For a fixed $\varepsilon,$ we have finally chosen $\eta$ so that $c2^{Q}%
\eta^{2}=\varepsilon^{2}$ and picked the corresponding $\delta$ depending on
$R_{0}$, $\mu$ and $\eta,$ that is on $R_{0},\mu,\varepsilon.$ This finishes
our proof.
\end{proof}

\begin{corollary}
\label{corollary 3.10}For any $\varepsilon>0$, $R_{0}>0,$ there is a small
$\delta=\delta\left(  \varepsilon,R_{0},\mu\right)  >0$ such that for any
$R\leq R_{0}/2,$ $z\in B_{R}\left(  \overline{x}\right)  $, $0<r\leq2R,$ if
$u$ is a weak solution of (\ref{system}) in $B_{11R}\left(  \overline
{x}\right)  \Subset\Omega$, the coefficients $a_{\alpha\beta}^{ij}\left(
x\right)  $ are $\left(  \delta,4r\right)  $-vanishing in $B_{R}\left(
\overline{x}\right)  $ and
\[
\left\vert \left\{  x\in B_{R}\left(  \overline{x}\right)  :\mathcal{M}%
_{B_{11R}\left(  \overline{x}\right)  }\left(  \left\vert Xu\right\vert
^{2}\right)  \left(  x\right)  >N_{1}^{2}\right\}  \cap B_{r}\left(  z\right)
\right\vert \geq\varepsilon\left\vert B_{r}\left(  z\right)  \right\vert ,
\]
then%
\begin{align*}
&  B_{r}\left(  z\right)  \cap B_{R}\left(  \overline{x}\right)  \\
&  \subset\left\{  x\in B_{R}\left(  \overline{x}\right)  :\mathcal{M}%
_{B_{11R}\left(  \overline{x}\right)  }\left(  \left\vert Xu\right\vert
^{2}\right)  \left(  x\right)  >1\right\}  \cup\left\{  x\in B_{R}\left(
\overline{x}\right)  :\mathcal{M}_{B_{11R}\left(  \overline{x}\right)
}\left(  \left\vert \mathbf{F}\right\vert ^{2}\right)  \left(  x\right)
>\delta^{2}\right\}  .
\end{align*}

\end{corollary}

\section{$L^{p}$ estimate on $\left\vert Xu\right\vert $%
\label{section L^p estimates}}

In this section we exploit the local estimates on the maximal function of
$\left\vert Xu\right\vert ^{2}$ proved in the previous section in order to
prove the desired $L^{p}$ bound. The starting point is the following useful
lemma about the estimate of the $L^{p}$ norm of a function by means of its
distribution function.

\begin{lemma}
\label{2.3}(See \cite[p.62]{Ca L}) Let $\theta>0$, $m>1$ be constants,
$p\in\left(  1,\infty\right)  $. Then there exists $c>0$ such that for any
nonnegative and measurable function $f$ in $\Omega$,%
\[
f\in L^{p}\left(  \Omega\right)  \text{ if and only if }S=\underset{l\geq
1}{\sum}m^{lp}\left\vert \left\{  x\in\Omega:f\left(  x\right)  >\theta
m^{l}\right\}  \right\vert <\infty
\]
and%
\[
\frac{1}{c}S\leq\left\Vert f\right\Vert _{L^{p}\left(  \Omega\right)  }%
^{p}\leq c\left(  \left\vert \Omega\right\vert +S\right)  .
\]

\end{lemma}

\begin{lemma}
[Vitali]\label{Vitali}Let $%
\mathcal{F}%
$ be a family of $d_{X}$-balls in $\mathbb{R}^{n}$ with bounded radii. There
exists a finite or countable sequence $\left\{  B_{i}\right\}  \subset%
\mathcal{F}%
$ of mutually disjoint balls such that%
\[
\bigcup_{B\in%
\mathcal{F}%
}B\subset\bigcup_{i}5B_{i}%
\]
where $5B$ is the ball with the same center as $B$ and radius five times big.
\end{lemma}

The proof is identical to that of the Euclidean case, with the Euclidean
distance replaced by $d_{X}$ here.

\begin{lemma}
\label{Lemma Vitali bis}Let $0<\varepsilon<1$, $C$ and $D$ be two measurable
sets satisfying $C\subset D\subset B_{R}\left(  \overline{x}\right)
\subset\Omega$, $\left\vert C\right\vert <\varepsilon\left\vert B_{R}\left(
\overline{x}\right)  \right\vert $ and the following property:%
\begin{equation}
\forall x\in B_{R}\left(  \overline{x}\right)  ,\forall r\leq2R,\left\vert
C\cap B_{r}\left(  x\right)  \right\vert \geq\varepsilon\left\vert
B_{r}\left(  x\right)  \right\vert \Longrightarrow B_{r}\left(  x\right)  \cap
B_{R}\left(  \overline{x}\right)  \subset D. \label{Hp Vitali}%
\end{equation}
Then
\[
\left\vert C\right\vert \leq\varepsilon\frac{5^{Q}}{c_{d}}\left\vert
D\right\vert ,
\]
where $c_{d}$ is the constants in (\ref{d-regular}).
\end{lemma}

\begin{proof}
For any $x\in C$, $C\subset B_{R}\left(  \overline{x}\right)  \subset
B_{2R}\left(  x\right)  $, hence%
\[
\left\vert C\cap B_{2R}\left(  x\right)  \right\vert =\left\vert C\right\vert
<\varepsilon\left\vert B_{R}\left(  \overline{x}\right)  \right\vert
<\varepsilon\left\vert B_{2R}\left(  x\right)  \right\vert .
\]
On the other hand, by Lebesgue differentiation theorem, for a.e. $x\in C,$%
\[
\lim_{r\rightarrow0}\frac{\left\vert C\cap B_{r}\left(  x\right)  \right\vert
}{\left\vert B_{r}\left(  x\right)  \right\vert }=1,
\]
hence for a.e. $x\in C$ there is an $r_{x}\leq2R$ such that for all
$r\in\left(  r_{x},2R\right)  $ it holds%
\begin{equation}
\left\vert C\cap B_{r_{x}}\left(  x\right)  \right\vert \geq\varepsilon
\left\vert B_{r_{x}}\left(  x\right)  \right\vert \text{ and }\left\vert C\cap
B_{r}\left(  x\right)  \right\vert <\varepsilon\left\vert B_{r}\left(
x\right)  \right\vert . \label{Vitali bis}%
\end{equation}
By Lemma \ref{Vitali}, there are $x_{1},$ $x_{2},\ldots\in C$, such that
$B_{r_{x_{1}}}\left(  x_{1}\right)  $, $B_{r_{x_{2}}}\left(  x_{2}\right)
,\ldots$ are mutually disjoint and satisfy%
\[%
{\displaystyle\bigcup\limits_{k}}
B_{5r_{x_{k}}}\left(  x_{k}\right)  \cap B_{R}\left(  \overline{x}\right)
\supset C.
\]

By (\ref{Vitali bis}) and (\ref{dimension}), we know
\[
\left\vert C\cap B_{5r_{x_{k}}}\left(  x_{k}\right)  \right\vert
<\varepsilon\left\vert B_{5r_{x_{k}}}\left(  x_{k}\right)  \right\vert
=\varepsilon5^{Q}\left\vert B_{r_{x_{k}}}\left(  x_{k}\right)  \right\vert .
\]
Also,
\begin{align*}
\left\vert C\right\vert  &  =\left\vert
{\displaystyle\bigcup\limits_{k}}
B_{5r_{x_{k}}}\left(  x_{k}\right)  \cap C\right\vert \leq\sum_{k}\left\vert
B_{5r_{x_{k}}}\left(  x_{k}\right)  \cap C\right\vert \\
&  \leq\varepsilon5^{Q}\sum_{k}\left\vert B_{r_{x_{k}}}\left(  x_{k}\right)
\right\vert \\
&  \leq\varepsilon\frac{5^{Q}}{c_{d}}\sum_{k}\left\vert B_{r_{x_{k}}}\left(
x_{k}\right)  \cap B_{R}\left(  \overline{x}\right)  \right\vert ,
\end{align*}
where the last inequality follows since $B_{R}\left(  \overline{x}\right)  $
is $d_{X}$-regular (see Remark \ref{d-regular and homegeneous}). Moreover
since the $B_{r_{x_{k}}}\left(  x_{k}\right)  $ are mutually disjoint the last
quantity equals%
\[
=\varepsilon\frac{5^{Q}}{c_{d}}\left\vert
{\displaystyle\bigcup\limits_{k}}
\left(  B_{r_{x_{k}}}\left(  x_{k}\right)  \cap B_{R}\left(  \overline
{x}\right)  \right)  \right\vert \leq\varepsilon\frac{5^{Q}}{c_{d}}\left\vert
D\right\vert ,
\]
since, by assumption (\ref{Hp Vitali}), $B_{r_{x_{k}}}\left(  x_{k}\right)
\cap B_{R}\left(  \overline{x}\right)  \subset D$. This completes the proof.
\end{proof}

\begin{theorem}
\label{Thm up to last}For any $\varepsilon>0$, $R_{0}>0$ there is a small
$\delta=\delta\left(  \varepsilon,R_{0},\mu\right)  >0$ such that for any
$R\leq R_{0}/2,$ if $u$ is a weak solution of (\ref{system}) in $B_{11R}%
\left(  \overline{x}\right)  \Subset\Omega$, the coefficients $a_{\alpha\beta
}^{ij}\left(  x\right)  $ are $\left(  \delta,8R\right)  $-vanishing in
$B_{R}\left(  \overline{x}\right)  $ and%
\begin{equation}
\left\vert \left\{  x\in B_{R}\left(  \overline{x}\right)  :\mathcal{M}%
_{B_{11R}\left(  \overline{x}\right)  }\left(  \left\vert Xu\right\vert
^{2}\right)  \left(  x\right)  >N_{1}^{2}\right\}  \right\vert <\varepsilon
\left\vert B_{R}\left(  \overline{x}\right)  \right\vert \label{Hp th fin}%
\end{equation}
(where $N_{1}$ is like in Theorem \ref{main lemma}), then for any positive
integer $m$,
\begin{align*}
&  \left\vert \left\{  x\in B_{R}\left(  \overline{x}\right)  :\mathcal{M}%
_{B_{11R}\left(  \overline{x}\right)  }\left(  \left\vert Xu\right\vert
^{2}\right)  \left(  x\right)  >N_{1}^{2m}\right\}  \right\vert \\
&  \leq\sum_{i=1}^{m}\varepsilon_{1}^{i}\left\vert \left\{  x\in B_{R}\left(
\overline{x}\right)  :\mathcal{M}_{B_{11R}\left(  \overline{x}\right)
}\left(  \left\vert \mathbf{F}\right\vert ^{2}\right)  \left(  x\right)
>\delta^{2}N_{1}^{2\left(  m-i\right)  }\right\}  \right\vert \\
&  +\varepsilon_{1}^{m}\left\vert \left\{  x\in B_{R}\left(  \overline
{x}\right)  :\mathcal{M}_{B_{11R}\left(  \overline{x}\right)  }\left(
\left\vert Xu\right\vert ^{2}\right)  \left(  x\right)  >1\right\}
\right\vert .
\end{align*}
where $\varepsilon_{1}=\varepsilon5^{Q}/c_{d}$.
\end{theorem}

\begin{proof}
Fix $\varepsilon,R_{0}>0$ and pick $\delta=\delta\left(  \varepsilon,R_{0}%
,\mu\right)  $ as in Corollary \ref{corollary 3.10}. We will prove this
assertion by induction on $m$. For $m=1$, we want to apply Lemma
\ref{Lemma Vitali bis} to
\begin{align*}
C  &  :=\left\{  x\in B_{R}\left(  \overline{x}\right)  :\mathcal{M}%
_{B_{11R}\left(  \overline{x}\right)  }\left(  \left\vert Xu\right\vert
^{2}\right)  \left(  x\right)  >N_{1}^{2}\right\}  ,\\
D  &  :=\left\{  x\in B_{R}\left(  \overline{x}\right)  :\mathcal{M}%
_{B_{11R}\left(  \overline{x}\right)  }\left(  \left\vert \mathbf{F}%
\right\vert ^{2}\right)  \left(  x\right)  >\delta^{2}\right\}  \cup\left\{
x\in B_{R}\left(  \overline{x}\right)  :\mathcal{M}_{B_{11R}\left(
\overline{x}\right)  }\left(  \left\vert Xu\right\vert ^{2}\right)  \left(
x\right)  >1\right\}  .
\end{align*}
Since $N_{1}\geq1$, $C\subset D\subset B_{R}\left(  \overline{x}\right)  $.
Also, by assumption $\left\vert C\right\vert <\varepsilon\left\vert
B_{R}\left(  \overline{x}\right)  \right\vert $. Let $x\in B_{R}\left(
\overline{x}\right)  $ such that
\[
\left\vert C\cap B_{r}\left(  x\right)  \right\vert \geq\varepsilon\left\vert
B_{r}\left(  x\right)  \right\vert .
\]
Then by Corollary \ref{corollary 3.10}
\[
B_{r}\left(  x\right)  \cap B_{R}\left(  \overline{x}\right)  \subset D,
\]
hence by Lemma \ref{Lemma Vitali bis}
\[
\left\vert C\right\vert \leq\varepsilon\frac{5^{Q}}{c_{d}}\left\vert
D\right\vert
\]
which is our assertion for $m=1$.

Now assume the assertion is valid for some $m$. Let $u$ be a weak solution to
(\ref{system}) in $B_{11R}\left(  \overline{x}\right)  $ satisfying
(\ref{Hp th fin}). Set $u_{1}=u/N_{1}$ and $\mathbf{F}_{1}=\mathbf{F}/N_{1}$,
then $u_{1}$ is a weak solution of
\[
X_{i}\left(  a_{\alpha\beta}^{ij}\left(  x\right)  X_{j}u_{1}\right)
=X_{i}\mathbf{F}_{1}%
\]
in $B_{11R}\left(  \overline{x}\right)  \Subset\Omega$, and satisfies%
\begin{align*}
&  \left\vert \left\{  x\in B_{R}\left(  \overline{x}\right)  :\mathcal{M}%
_{B_{11R}\left(  \overline{x}\right)  }\left(  \left\vert Xu_{1}\right\vert
^{2}\right)  \left(  x\right)  >N_{1}^{2}\right\}  \right\vert \\
&  =\left\vert \left\{  x\in B_{R}\left(  \overline{x}\right)  :\mathcal{M}%
_{B_{11R}\left(  \overline{x}\right)  }\left(  \left\vert Xu\right\vert
^{2}\right)  \left(  x\right)  >N_{1}^{4}\right\}  \right\vert \\
&  <\left\vert \left\{  x\in B_{R}\left(  \overline{x}\right)  :\mathcal{M}%
_{B_{11R}\left(  \overline{x}\right)  }\left(  \left\vert Xu\right\vert
^{2}\right)  \left(  x\right)  >N_{1}^{2}\right\}  \right\vert <\varepsilon
\left\vert B_{R}\left(  \overline{x}\right)  \right\vert .
\end{align*}
By the induction assumption on $m$, we have%
\begin{align}
&  \left\vert \left\{  x\in B_{R}\left(  \overline{x}\right)  :\mathcal{M}%
_{B_{11R}\left(  \overline{x}\right)  }\left(  \left\vert Xu\right\vert
^{2}\right)  \left(  x\right)  >N_{1}^{2\left(  m+1\right)  }\right\}
\right\vert \nonumber\\
&  =\left\vert \left\{  x\in B_{R}\left(  \overline{x}\right)  :\mathcal{M}%
_{B_{11R}\left(  \overline{x}\right)  }\left(  \left\vert Xu_{1}\right\vert
^{2}\right)  \left(  x\right)  >N_{1}^{2m}\right\}  \right\vert \nonumber\\
&  \leq\sum_{i=1}^{m}\varepsilon_{1}^{i}\left\vert \left\{  x\in B_{R}\left(
\overline{x}\right)  :\mathcal{M}_{B_{11R}\left(  \overline{x}\right)
}\left(  \left\vert \mathbf{F}_{1}\right\vert ^{2}\right)  \left(  x\right)
>\delta^{2}N_{1}^{2\left(  m-i\right)  }\right\}  \right\vert \nonumber\\
&  +\varepsilon_{1}^{m}\left\vert \left\{  x\in B_{R}\left(  \overline
{x}\right)  :\mathcal{M}_{B_{11R}\left(  \overline{x}\right)  }\left(
\left\vert Xu_{1}\right\vert ^{2}\right)  \left(  x\right)  >1\right\}
\right\vert \nonumber\\
&  =\sum_{i=1}^{m}\varepsilon_{1}^{i}\left\vert \left\{  x\in B_{R}\left(
\overline{x}\right)  :\mathcal{M}_{B_{11R}\left(  \overline{x}\right)
}\left(  \left\vert \mathbf{F}\right\vert ^{2}\right)  \left(  x\right)
>\delta^{2}N_{1}^{2\left(  m+1-i\right)  }\right\}  \right\vert \nonumber\\
&  +\varepsilon_{1}^{m}\left\vert \left\{  x\in B_{R}\left(  \overline
{x}\right)  :\mathcal{M}_{B_{11R}\left(  \overline{x}\right)  }\left(
\left\vert Xu\right\vert ^{2}\right)  \left(  x\right)  >N_{1}^{2}\right\}
\right\vert . \label{fin 1}%
\end{align}
On the other hand, by the assertion valid for $m=1,$%
\begin{align}
&  \left\vert \left\{  x\in B_{R}\left(  \overline{x}\right)  :\mathcal{M}%
_{B_{11R}\left(  \overline{x}\right)  }\left(  \left\vert Xu\right\vert
^{2}\right)  \left(  x\right)  >N_{1}^{2}\right\}  \right\vert \nonumber\\
&  \leq\varepsilon_{1}\left\vert \left\{  x\in B_{R}\left(  \overline
{x}\right)  :\mathcal{M}_{B_{11R}\left(  \overline{x}\right)  }\left(
\left\vert \mathbf{F}\right\vert ^{2}\right)  \left(  x\right)  >\delta
^{2}\right\}  \right\vert \nonumber\\
&  +\varepsilon_{1}\left\vert \left\{  x\in B_{R}\left(  \overline{x}\right)
:\mathcal{M}_{B_{11R}\left(  \overline{x}\right)  }\left(  \left\vert
Xu\right\vert ^{2}\right)  \left(  x\right)  >1\right\}  \right\vert .
\label{fin 2}%
\end{align}
Putting (\ref{fin 2}) into (\ref{fin 1}) we get%
\[
\left\vert \left\{  x\in B_{R}\left(  \overline{x}\right)  :\mathcal{M}%
_{B_{11R}\left(  \overline{x}\right)  }\left(  \left\vert Xu\right\vert
^{2}\right)  \left(  x\right)  >N_{1}^{2\left(  m+1\right)  }\right\}
\right\vert
\]%
\begin{align*}
&  \leq\sum_{i=1}^{m}\varepsilon_{1}^{i}\left\vert \left\{  x\in B_{R}\left(
\overline{x}\right)  :\mathcal{M}_{B_{11R}\left(  \overline{x}\right)
}\left(  \left\vert \mathbf{F}\right\vert ^{2}\right)  \left(  x\right)
>\delta^{2}N_{1}^{2\left(  m+1-i\right)  }\right\}  \right\vert \\
&  +\varepsilon_{1}^{m+1}\left\vert \left\{  x\in B_{R}\left(  \overline
{x}\right)  :\mathcal{M}_{B_{11R}\left(  \overline{x}\right)  }\left(
\left\vert Xu\right\vert ^{2}\right)  \left(  x\right)  >1\right\}
\right\vert \\
&  +\varepsilon_{1}^{m+1}\left\vert \left\{  x\in B_{R}\left(  \overline
{x}\right)  :\mathcal{M}_{B_{11R}\left(  \overline{x}\right)  }\left(
\left\vert \mathbf{F}\right\vert ^{2}\right)  \left(  x\right)  >\delta
^{2}\right\}  \right\vert
\end{align*}%
\begin{align*}
&  =\sum_{i=1}^{m+1}\varepsilon_{1}^{i}\left\vert \left\{  x\in B_{R}\left(
\overline{x}\right)  :\mathcal{M}_{B_{11R}\left(  \overline{x}\right)
}\left(  \left\vert \mathbf{F}\right\vert ^{2}\right)  \left(  x\right)
>\delta^{2}N_{1}^{2\left(  m+1-i\right)  }\right\}  \right\vert \\
&  +\varepsilon_{1}^{m+1}\left\vert \left\{  x\in B_{R}\left(  \overline
{x}\right)  :\mathcal{M}_{B_{11R}\left(  \overline{x}\right)  }\left(
\left\vert Xu\right\vert ^{2}\right)  \left(  x\right)  >1\right\}
\right\vert
\end{align*}
which is the desired assertion for $m+1$. This completes the proof.
\end{proof}

We can finally come to the

\bigskip

\begin{proof}
[Proof of Theorem \ref{Thm main basic}]Fix $R_{0}$, let $\varepsilon>0$ to be
chosen later, and pick $\delta=\delta\left(  \varepsilon,R_{0},\mu\right)  $
as in Theorem \ref{Thm up to last}. For $\lambda>0,$ let $u_{\lambda}=\frac
{u}{\lambda},\mathbf{F}_{\lambda}=\frac{\mathbf{F}}{\lambda}$. We claim that
we can take $\lambda$ large enough (depending on $\varepsilon,u$ and
$\mathbf{F}$) so that%
\begin{equation}
\left\vert \left\{  x\in B_{R}\left(  \overline{x}\right)  :\mathcal{M}%
_{B_{11R}\left(  \overline{x}\right)  }\left(  \left\vert Xu_{\lambda
}\right\vert ^{2}\right)  \left(  x\right)  >N_{1}^{2}\right\}  \right\vert
<\varepsilon\left\vert B_{R}\left(  \overline{x}\right)  \right\vert
\label{3.31}%
\end{equation}
and
\begin{equation}
\sum_{k=1}^{\infty}N_{1}^{kp}\left\vert \left\{  x\in B_{R}\left(
\overline{x}\right)  :\mathcal{M}_{B_{11R}\left(  \overline{x}\right)
}\left(  \left\vert \mathbf{F}_{\lambda}\right\vert ^{2}\right)  \left(
x\right)  >\delta N_{1}^{2k}\right\}  \right\vert \leq1. \label{3.32}%
\end{equation}
Actually, since $\mathbf{F}\in L^{p}\left(  B_{11R}\left(  \overline
{x}\right)  ;\mathbb{M}^{N\times q}\right)  $ with $p>2$, we have
$\mathcal{M}_{B_{11R}\left(  \overline{x}\right)  }\left(  \left\vert
\mathbf{F}_{\lambda}\right\vert ^{2}\right)  \left(  x\right)  \in L^{\frac
{p}{2}}\left(  B_{11R}\left(  \overline{x}\right)  \right)  $ by Lemma
\ref{maximal function}. Applying Lemma \ref{2.3} with $f=$ $\mathcal{M}%
_{B_{11R}\left(  \overline{x}\right)  }\left(  \left\vert \mathbf{F}_{\lambda
}\right\vert ^{2}\right)  $, $\theta=\delta$, $m=N_{1}^{2}$, $\Omega
=B_{R}\left(  \overline{x}\right)  $ and $p$ replaced by $p/2$, there is a
positive constant $c$ depending only on $\delta,p$ and $N_{1}$, such that%
\begin{align*}
&  \sum_{k=1}^{\infty}N_{1}^{kp}\left\vert \left\{  x\in B_{R}\left(
\overline{x}\right)  :\mathcal{M}_{B_{11R}\left(  \overline{x}\right)
}\left(  \left\vert \mathbf{F}_{\lambda}\right\vert ^{2}\right)  \left(
x\right)  >\delta N_{1}^{2k}\right\}  \right\vert \\
&  \leq c\left\Vert \mathcal{M}_{B_{11R}\left(  \overline{x}\right)  }\left(
\left\vert \mathbf{F}_{\lambda}\right\vert ^{2}\right)  \right\Vert
_{L^{p/2}\left(  B_{11R}\left(  \overline{x}\right)  \right)  }^{p/2}\leq
c\left\Vert \mathbf{F}_{\lambda}\right\Vert _{L^{p}\left(  B_{11R}\left(
\overline{x}\right)  \right)  }^{p}.
\end{align*}
Also, by Lemma \ref{maximal function} we have%
\begin{align*}
&  \left\vert \left\{  x\in B_{R}\left(  \overline{x}\right)  :\mathcal{M}%
_{B_{11R}\left(  \overline{x}\right)  }\left(  \left\vert Xu_{\lambda
}\right\vert ^{2}\right)  \left(  x\right)  >N_{1}^{2}\right\}  \right\vert \\
&  =\left\vert \left\{  x\in B_{R}\left(  \overline{x}\right)  :\mathcal{M}%
_{B_{11R}\left(  \overline{x}\right)  }\left(  \left\vert Xu\right\vert
^{2}\right)  \left(  x\right)  >\lambda^{2}N_{1}^{2}\right\}  \right\vert
\leq\frac{c}{\lambda^{2}N_{1}^{2}}\left\Vert Xu\right\Vert _{L^{2}\left(
B_{11R}\left(  \overline{x}\right)  \right)  }^{2}%
\end{align*}
Hence we can take
\begin{equation}
\lambda=c\left(  \frac{\left\Vert Xu\right\Vert _{L^{2}\left(  B_{11R}\left(
\overline{x}\right)  ;\mathbb{R}^{N}\right)  }}{\varepsilon^{1/2}\left\vert
B_{R}\left(  \overline{x}\right)  \right\vert ^{1/2}}+\left\Vert
\mathbf{F}\right\Vert _{L^{p}\left(  B_{11R}\left(  \overline{x}\right)
\right)  }\right)  \label{lambda}%
\end{equation}
for some constant $c$ depending on $\delta,p,N_{1},$ hence $c=\left(
\varepsilon,R_{0},p,\mathbb{G}\right)  $, and get (\ref{3.31}) and
(\ref{3.32}) satisfied.

Next, by (\ref{3.31}) we can apply Theorem \ref{Thm up to last} to
$u_{\lambda}$ for this large $\lambda$, writing%
\begin{align*}
&  \sum_{k=1}^{\infty}N_{1}^{kp}\left\vert \left\{  x\in B_{R}\left(
\overline{x}\right)  :\mathcal{M}_{B_{11R}\left(  \overline{x}\right)
}\left(  \left\vert Xu_{\lambda}\right\vert ^{2}\right)  \left(  x\right)
>N_{1}^{2k}\right\}  \right\vert \\
&  \leq\sum_{k=1}^{\infty}N_{1}^{kp}\left(  \sum_{i=1}^{k}\varepsilon_{1}%
^{i}\left\vert \left\{  x\in B_{R}\left(  \overline{x}\right)  :\mathcal{M}%
_{B_{11R}\left(  \overline{x}\right)  }\left(  \left\vert \mathbf{F}_{\lambda
}\right\vert ^{2}\right)  \left(  x\right)  >\delta^{2}N_{1}^{2\left(
k-i\right)  }\right\}  \right\vert \right.  \\
&  \left.  +\varepsilon_{1}^{k}\left\vert \left\{  x\in B_{R}\left(
\overline{x}\right)  :\mathcal{M}_{B_{11R}\left(  \overline{x}\right)
}\left(  \left\vert Xu_{\lambda}\right\vert ^{2}\right)  \left(  x\right)
>1\right\}  \right\vert \right)  \\
&  =\sum_{i=1}^{\infty}\left(  N_{1}^{p}\varepsilon_{1}\right)  ^{i}\sum
_{k=i}^{\infty}N_{1}^{p\left(  k-i\right)  }\left\vert \left\{  x\in
B_{R}\left(  \overline{x}\right)  :\mathcal{M}_{B_{11R}\left(  \overline
{x}\right)  }\left(  \left\vert \mathbf{F}_{\lambda}\right\vert ^{2}\right)
\left(  x\right)  >\delta^{2}N_{1}^{2\left(  k-i\right)  }\right\}
\right\vert \\
+ &  \sum_{i=1}^{\infty}\left(  N_{1}^{p}\varepsilon_{1}\right)
^{i}\left\vert \left\{  x\in B_{R}\left(  \overline{x}\right)  :\mathcal{M}%
_{B_{11R}\left(  \overline{x}\right)  }\left(  \left\vert Xu_{\lambda
}\right\vert ^{2}\right)  \left(  x\right)  >1\right\}  \right\vert
\end{align*}
by (\ref{3.32})%
\[
=\sum_{i=1}^{\infty}\left(  N_{1}^{p}\varepsilon_{1}\right)  ^{i}\left(
1+\left\vert B_{R}\left(  \overline{x}\right)  \right\vert \right)
<1+\left\vert B_{R}\left(  \overline{x}\right)  \right\vert
\]
taking $\varepsilon$ so that $N_{1}^{p}\varepsilon_{1}=1/2$. We have finally
chosen $\varepsilon$ small enough, depending on $p$ and $\mathbb{G}$, and a
corresponding $\delta=\delta\left(  \varepsilon,R_{0},\mu\right)
=\delta\left(  p,\mathbb{G},R_{0},\mu\right)  .$

Therefore we can apply Lemma \ref{2.3} to $f=\mathcal{M}_{B_{11R}\left(
\overline{x}\right)  }\left(  \left\vert Xu_{\lambda}\right\vert ^{2}\right)
\left(  x\right)  $ and $m=N_{1}^{2}$ getting%
\[
\left\Vert \mathcal{M}_{B_{11R}\left(  \overline{x}\right)  }\left(
\left\vert Xu_{\lambda}\right\vert ^{2}\right)  \right\Vert _{L^{p/2}\left(
B_{R}\left(  \overline{x}\right)  \right)  }^{p/2}\leq c\left(  1+R^{Q}%
\right)
\]
with $c=c\left(  p,\mathbb{G}\right)  $, which by (\ref{lambda}) implies
\[
\left\Vert \mathcal{M}_{B_{11R}\left(  \overline{x}\right)  }\left(
\left\vert Xu\right\vert ^{2}\right)  \right\Vert _{L^{p/2}\left(
B_{R}\left(  \overline{x}\right)  \right)  }^{1/2}\leq c\left\{  \left\Vert
Xu\right\Vert _{L^{2}\left(  B_{11R}\left(  \overline{x}\right)  \right)
}+\left\Vert \mathbf{F}\right\Vert _{L^{p}\left(  B_{11R}\left(  \overline
{x}\right)  \right)  }\right\}
\]
with $c=$ $c\left(  R,R_{0},p,\mathbb{G}\right)  ,$ and recalling that
$\left\vert f\left(  x\right)  \right\vert \leq\mathcal{M}_{B_{11R}\left(
\overline{x}\right)  }\left(  f\right)  \left(  x\right)  $ for a.e. $x$, we
get%
\[
\left\Vert Xu\right\Vert _{L^{p}\left(  B_{R}\left(  \overline{x}\right)
\right)  }\leq c\left\{  \left\Vert Xu\right\Vert _{L^{2}\left(
B_{11R}\left(  \overline{x}\right)  \right)  }+\left\Vert \mathbf{F}%
\right\Vert _{L^{p}\left(  B_{11R}\left(  \overline{x}\right)  \right)
}\right\}  .
\]
This completes the proof of Theorem \ref{Thm main basic}.
\end{proof}

\bigskip

Maochun Zhu and Pengcheng Niu

Department of Applied Mathematics,

Northwestern Polytechnical University

Xi'an, Shaanxi, 710129, CHINA

zhumaochun2006@126.com

pengchengniu@nwpu.edu.cn

\bigskip

Marco Bramanti

Dipartimento di Matematica,

Politecnico di Milano

Via Bonardi 9. 20133 Milano, ITALY

marco.bramanti@polimi.it

\end{document}